\documentclass[11pt,letterpaper]{amsart}
\usepackage{color}
\usepackage{amsmath}
\usepackage{todonotes}
\usepackage{amssymb}
\usepackage{amsthm}
\usepackage[msc-links,abbrev]{amsrefs}
\usepackage{hyperref}
\usepackage[noabbrev,capitalize]{cleveref}
\usepackage[english]{babel}
\usepackage{IEEEtrantools}
\usepackage{mathrsfs}
\usepackage{mathtools}
\usepackage{enumitem}
\newcommand{\vt}{(V_\phi)}
\newcommand{\cphi}{\widetilde{\phi}}
\newcommand{\mG}{\mathcal{G}}

\newcommand{\pt}{(P_{\phi})}
\newcommand{\cmG}{\widetilde{\mathcal{G}}}
\newcommand{\bR}{\mathbb{R}}
\newcommand{\cg}{\widetilde{g}}

\newcommand{\manlam}{(M^n,g,\phi,\lambda)}

\newcommand{\subba}{(4\pi\tau)^{-n/2}e^{-\phi}\mathrm{dvol}_g}

\newcommand{\amb}{(\widetilde{\mathcal{G}},\widetilde{g},\widetilde{\phi},\lambda)}
\newcommand{\man}{(M^n,g,\phi,\lambda)}

\newcommand\restr[2]{{
  \left.\kern-\nulldelimiterspace 
  #1 
  \vphantom{\big|} 
  \right|_{#2} 
  }}

\def\sideremark#1{\ifvmode\leavevmode\fi\vadjust{\vbox to0pt{\vss
 \hbox to 0pt{\hskip\hsize\hskip1em
 \vbox{\hsize3cm\tiny\raggedright\pretolerance10000
 \noindent #1\hfill}\hss}\vbox to8pt{\vfil}\vss}}}

\newtheorem{theorem}{Theorem}[section]

\newtheorem{lemma}[theorem]{Lemma}
\newtheorem{corollary}[theorem]{Corollary}

\theoremstyle{definition}
\newtheorem{definition}[theorem]{Definition}

\newtheorem*{example}{Example}

\theoremstyle{remark}

\numberwithin{equation}{section}

\title{The weighted ambient metric for manifolds with density}
\author{Ayush Khaitan}
\address{Department of Mathematics \\ Rutgers University \\ New Brunswick, NJ 08854 \\ USA}
\email{ayush.khaitan@rutgers.edu}

\begin{document}
\keywords{ambient metric; manifold with density; renormalized volume coefficient; GJMS operator, $\mathcal{F}$--functional, $\mathcal{W}$--functional, Ricci flow, Kleiner-Lott spacetime, singular Ricci flow}
\subjclass[2020]{Primary 53C18; Secondary 53A31, 58E30}

\maketitle
\begin{abstract}
We prove the existence and uniqueness of a weighted analogue of the Fefferman-Graham ambient metric for manifolds with density.  We then show that this ambient metric forms the natural geometric framework for the Ricci flow by constructing infinite families of fully non-linear analogues of Perelman's $\mathcal{F}$ and $\mathcal{W}$ functionals. We extend Perelman's monotonicity result to these two families of functionals under several conditions, including for shrinking solitons and Einstein manifolds. We do so by constructing a ``Ricci flow vector field" in the ambient space, which may be of independent research interest. We also prove that the weighted GJMS operators associated with the weighted ambient metric are formally self-adjoint,  and that the associated weighted renormalized volume coefficients are variational.
\end{abstract}

\section{Introduction}
The Fefferman--Graham ambient space~\cite{FeffermanGraham2012} is a formally Ricci flat space canonically associated to a given conformal manifold.
Among its many applications are the classification of local scalar conformal invariants~\cites{BaileyEastwoodGraham1994,Gover2001} and the construction of a family of conformally covariant operators, called GJMS operators, with leading-order term a power of the Laplacian~\cite{GJMS1992}.

A \emph{manifold with density} $\man$ is a Riemannian manifold with a smooth measure $e^{-\phi}\mathrm{dvol}_g$, and may be thought of as a smooth metric measure space with $m=\infty$. Among its applications is the realization of the $\mathcal{F}$ and $\mathcal{W}$ gradient Ricci flows~\cite{Perelman1}. A manifold with density $\man$ is a \emph{gradient Ricci soliton} if $\mathrm{Ric}+\nabla^2\phi=\lambda g$. Here $\lambda$ is a parameter that equals $1/2\tau$, where $\tau=T-t$ may be thought of as ``backwards time".  

Chang, Gursky and Yang~\cite{ChangGurskyYang2006} started a program to construct conformal invariants associated to a manifold with density. Case extended this program~\cites{Case2016v,Case2014sd} by constructing two higher fully non-linear analogues of the $\mathcal{W}$-functional, $\mathcal{W}_{2,\phi}$ and $\mathcal{W}_{3,\phi}$, that are also conformally invariant, and proving that they are variational. He mentions~\cite{Case2016v} that a systematic way to construct these invariants would be by constructing the ambient metric for manifolds with density. Case and the author tried to construct the ambient metric for a manifold with density by constructing one for $m<\infty$~\cite{CaseKhaitan2022}, and then attempting to take a limit as $m\to\infty$. However, their attempt was unsuccessful, as the geometric quantities that they construct are not always well-defined when $m=\infty$ (one example being the volume element $f^m \mathrm{dvol}_{g_\rho}$).

In this paper, we finally construct the ambient metric for manifolds with density, thereby bringing this program to a completion. This allows us to systematically construct infinite families of fully non-linear analogues of both the $\mathcal{F}$ and $\mathcal{W}$ functionals. We prove that these families are variational, and also construct the associated weighted GJMS operators, and weighted renormalized volume coefficients. 

Moreover, we establish a dual correspondence between the weighted ambient metric and the singular Ricci flow by proving that given a singular gradient Ricci flow spacetime $\mathcal{H}$ in the Kleiner-Lott sense~\cites{KleinerLott2017}, we can construct a unique global ambient half-space structure on $\mathcal{H}\times \mathbb{R}$. We also prove that every global ambient space contains a singular Ricci flow spacetime, thereby completing the correspondence between the two. 

The main application of our results is  the construction of infinite families of fully non-linear analogues of Perelman's $\mathcal{F}$--functional and $\mathcal{W}$--functional. We also prove that, under suitable conditions, these families of functionals are monotone increasing or decreasing under the singular Ricci flow. We prove that the renormalized volume coefficients, which are the integrands of these fully non-linear analogues of Perelman's functionals, are variational, and prove that shrinking solitons are stable points of the $\mathcal{W}$ functional in a conformal class.

Our constructions require us to construct a vector field in the weighted ambient space, along which we have a solution to the gradient Ricci flow. By the use of this vector field, we are also able to provide short proofs of the monotonicity of Perelman's $\mathcal{F}$ and $\mathcal{W}$ functionals that, to our knowledge, are not present in the literature. See \cref{f-increasing,w-increasing} for more details. This vector field may be of independent research interest.

Before stating our theorems, we recall some well-known weighted invariants of a manifold with density $\man$~\cite{Case2014sd}:\begin{align*}
\mathrm{Ric}_\phi&=\mathrm{Ric}+\nabla^2\phi,\\
R_{\phi}&=R+2\Delta\phi-|\nabla\phi|^2+2\lambda(\phi-n)
.
\end{align*}
Also, consider the following weighted invariants that will be useful in this paper~\cites{Case2014s,CaseChang2013}
\begin{equation}
\label{basic-definitions-2}
\begin{aligned}
P_{\phi}&=\mathrm{Ric}_\phi-\lambda g,\\
F_\phi&=\Delta\phi-|\nabla\phi|^2+2\lambda \phi-n\lambda,\\
Y_{\phi}&=-\frac{1}{2}[R+|\nabla\phi|^2-2\lambda\phi].
\end{aligned}
\end{equation}
We also require the weighted Bach tensor, which is defined as~\cites{Case2014sd,Case2016v}
\begin{equation*}
\label{weighted-bach}
{B}_\phi=\delta_\phi d {P}_\phi+\mathrm{Rm} \cdot {P}_\phi.
\end{equation*}
where
\begin{equation*}
\begin{aligned}
& \left(\delta_\phi d {P}_\phi\right)(X, Y):=\sum_{j=1}^n \nabla_{E_i} d {P}_\phi\left(E_i, X, Y\right)-d {P}_\phi(\nabla \phi, X, Y), \\
& \left(\operatorname{Rm} \cdot {P}_\phi\right)(X, Y):=\sum_{j=1}^n \operatorname{Rm}\left(E_i, X, E_j, Y\right) {P}_\phi\left(E_i, E_j\right). \\
&
\end{aligned}
\end{equation*}
More precise definitions of all these terms are given in \cref{sec:manifolds-with-density,sec:ambient-metric}. \\

A \emph{gradient Ricci soliton} is a manifold with density $\manlam$ that satisfies $P_\phi=0$. If $M$ is connected, as we will always assume, then there is a constant $c\in\mathbb{R}$ such that $F_\phi=c$. Putting these together, we get
\begin{equation*}
\begin{aligned}
P_\phi&=0,\\
F_\phi&=c.    
\end{aligned}    
\end{equation*}
Given a metric measure structure $(g,\phi)$, the weighted conformal class $\mathcal{C}$ of $(g,\phi)$ is defined as 
\begin{equation}
\label{conformal-class}
\mathcal{C}:=\{\phi+\omega\mid \omega\in C^\infty(M)\}.
\end{equation} We denote the weighted conformal class of $(g,\phi)$ as $[\phi]$. Let $\mathcal{G}$ consist of all pairs $(\psi,x)$, where $x\in M$, and $\psi\in C^\infty(M)$ such that $\psi=\phi+\omega$ for some $\omega\in C^\infty(M)$. For $s\in \mathbb{R}$, we define the dilation $\delta_s:\mathcal{G}\to \mathcal{G}$ as $(\psi,x)\to (\psi+s,x)$, and the projection map $\pi:\mathcal{G}\to M$ as $(\psi,x)\to x$. Let $\widetilde{\mathcal{G}}$ denote the space $\mathcal{G}\times (-\epsilon,\epsilon)$. We embed $\mathcal{G}\to \widetilde{\mathcal{G}}$ by $\iota: z\to (z,0)$. Roughly speaking, a weighted ambient space for $(M^n,g,[\phi],\lambda)$ is a manifold with density $(\mathbb{R}\times M^n\times (-\epsilon,\epsilon),\widetilde{g},\widetilde{\phi},\lambda)$ of Lorentzian signature such that
\begin{align*}
\restr{\widetilde{P}_\phi}{u=0}&=0\;\bmod\;O(u^\infty),\\
\restr{\widetilde{F}_\phi}{u=0}&=0\;\bmod\;O(u^\infty),
\end{align*}
where $u$ denotes the coordinate on $(-\epsilon,\epsilon)$. 

Pre-ambient spaces $(\cmG_1,\widetilde{g},\widetilde{\phi},\lambda_1)$ and $(\cmG_2,\widetilde{g}_2,\widetilde{\phi}_2,\lambda_2)$ are considered to be \emph{ambient-equivalent} if there exists a diffeomorphism $\psi:\cmG_1\to \cmG_2$ which commutes with dilations, satisfies $\restr{\psi}{\mG}=\mathrm{id}$, and is such that $(\widetilde{g}_1-\psi^*\widetilde{g}_2,\widetilde{\phi}_1-\psi^*\widetilde{\phi}_2)=O(u^\infty)$.

Given a weighted conformal class of manifolds with density, we show that there exists a weighted ambient metric for manifolds with density that is unique up to ambient-equivalence. 
\begin{theorem}
\label{ambient-metric-theorem}
Given a weighted conformal class of manifolds with density $(M^n,g,[\phi],\lambda)$, there exists a unique, up to ambient-equivalence, weighted ambient space $\amb$ for it.
\end{theorem}
From this point onwards, we will assume that the metric measure structure $(\widetilde{g},\widetilde{\phi})$ is of the form 
\begin{equation}
\label{normalambientmetric}
\begin{aligned}
\widetilde{g}=2dudv+g_{u}, \quad \widetilde{\phi}=\phi_{u}-v+\lambda uv,
\end{aligned}   
\end{equation}
where $v$ is the coordinate on $\mathbb{R}$, and $(g_u,\phi_u)$ is a $1$-parameter family of metric measure structures such that $\restr{(g_u,\phi_u)}{u=0}=(g,\phi)$. In this paper, we refer to the $v$ coordinate as the $0$ coordinate, the $u$ coordinate as the $\infty$ coordinate, and the coordinates on the manifold $M$ as lower case Roman letter $i,j,k$ (c.f.~\cites{CaseKhaitan2022,FeffermanGraham2012}).
Metrics of the form of \cref{normalambientmetric} are called \emph{normal}, and all other ambient metrics are ambient-equivalent to this metric. See \cref{sec:ambient-metric} for more details. 

We can explicitly identify the weighted ambient metric when the underlying manifold with density is a gradient Ricci soliton, or $g$ is flat.
\begin{theorem}
 \label{examples-theorem}
Let $\man$ be manifold with density such that 
\begin{enumerate}[label=(\roman*)]
\item \label{Theorem 1.2-1} $\manlam$ is a gradient Ricci soliton; or

\item \label{Theorem 1.2-2} $g$ is flat, or equivalently that $\manlam$ is locally conformally flat.
\end{enumerate}
Set 
\begin{equation}
\label{quadratic-expression}
\begin{aligned}
\left(g_{u}\right)_{i j} &=g_{i j}+2u(P_{\phi})_{i j} +u^{2}(P_{\phi})_{i k}(P_{\phi})_{j}^{k} , \\
\phi_{u} &=\phi- u Y_{\phi}.
\end{aligned}
\end{equation}
Then $(\widetilde{\mathcal{G}},\widetilde{g},\widetilde{\phi},\lambda)$ with $(\widetilde{g},\widetilde{\phi})$ of the form of \cref{normalambientmetric} is a weighted ambient space for $\man$.
\end{theorem}

In ~\cites{Case2014sd,Case2016v}, Case proposes a natural definition of the first three weighted renormalized volume coefficients for manifolds with density, and proves that they are variational. In this paper, we extend these results to all renormalized volume coefficients. Given a normal ambient metric of the form of \cref{normalambientmetric}, we define the \emph{renormalized volume} $v_\phi$ as 
\[v_\phi:=\frac{e^{-\widetilde{\phi}}}{e^{-\phi}}\left(\frac{\operatorname{det} g_{u}}{\operatorname{det} g}\right)^{1 / 2},\]
and the \emph{renormalized volume coefficients} for manifolds with density as
\begin{equation}
\label{volcoefficientsdef}
v_{k,\phi}:=\frac{1}{k!}\restr{\partial_u^k}{u,v=0}v_\phi, \end{equation}
where $\partial_u^k$ denotes $\underbrace{\partial_u\dots \partial_u}_{k\text{ times}}$.
We have 
\begin{equation}
\begin{aligned}
\label{v-formulas}
v_{1,\phi}&=\frac{1}{2}R_{\phi},\\
v_{2,\phi}&=\frac{1}{2}((v_{1,{\phi}})^2-|P_\phi|^2),\\
v_{3,\phi}&=\frac{1}{6}(v_{1,\phi}^3-3v_{1,\phi}|P_\phi|^2+2(P_\phi)^3)+\frac{1}{3}\langle P_\phi,B_\phi\rangle.
\end{aligned}
\end{equation}
These recover the formulas for $v_{1,\phi},v_{2,\phi}$ and $v_{3,\phi}$ in~\cites{Case2014sd,Case2016v}.

We can explicitly write down the ambient metric and renormalized volume coefficients for Einstein metrics when $\lambda=0$. With their help, we later study the evolution of the generalized $\mathcal{F}$--functional under the Ricci flow. See \cref{ricci-flow} below.
\begin{theorem}
 \label{second-examples-theorem}
For Einstein manifolds $(M^n,g,\phi,\lambda=0)$ that satisfy 
\begin{equation}
\label{einstein-manifold-equation}
\mathrm{Ric}=2\mu g,\quad \phi=c
\end{equation}
for some $c\in\mathbb{R}$, the normal ambient metric $(\widetilde{g},\widetilde{\phi})$ of the form of \cref{normalambientmetric} satisfies
\begin{equation}
\label{einstein-ambient-metric-2}
g_u=(1+4\mu u)g,\quad \phi_u=\frac{n}{4}\mathrm{ln}|1+4\mu u|+c.   \end{equation}
Moreover, we deduce from \cref{einstein-ambient-metric-2} that
\begin{equation}
\label{einstein-volume-coefficients}
    v_{k,\phi}=\frac{1}{k!}\mu^k \prod_{j=0}^{k-1}(n-4j).
\end{equation}
\end{theorem}
The equations of $\mathcal{F}$-gradient Ricci flow are~\cite[(1.1)]{Perelman1}
\begin{equation*}
\begin{aligned}
\partial_t g&=-2\mathrm{Ric}_\phi,\\
\partial_t \phi&=-R-\Delta\phi.
\end{aligned}
\end{equation*}

Similarly, the equations of $\mathcal{W}$--gradient Ricci flow are~\cite{Perelman1}*{(3.3)}
\begin{equation*}
\begin{aligned}
 \partial_t g_{ij}&=-2(\mathrm{Ric}+\nabla^2\phi),\\
 \partial_t \phi&=-R-\Delta\phi+\frac{ n}{2\tau},\\
 \partial_t \tau&=-1.
\end{aligned}
\end{equation*}
Note that the $\mathcal{F}$ and $\mathcal{W}$ gradient Ricci flows are equivalent to the Ricci flow via the pullback of $(g,\phi)$ with respect to the one-parameter family of diffeomorphisms $\psi_t$, generated by $\nabla\phi$~\cites{Perelman1}. 

We have a formal solution to the $\mathcal{F}$--gradient Ricci flow in the ambient space $(\widetilde{\mathcal{G}},\widetilde{g},\widetilde{\phi},\lambda=0)$ along the vector field $(v_{1,\phi})_u\partial_v-\partial_u$, where $(v_{1,\phi})_u:=\frac{1}{2}(R_u+2\Delta_u\phi_u-|\nabla_u\phi_u|^2)$~(cf.\ \cref{v-formulas}). We also have a solution to the $\mathcal{W}$--gradient Ricci flow, but it exists in a transverse direction to the ambient space $(\widetilde{\mathcal{G}},\widetilde{g},\widetilde{\phi},\lambda)$ for $\lambda\neq 0$, and will be explained later in this section.
\begin{theorem}
\label{ricci-flow}
Given a manifold with density $\man$ with $\lambda=0$, we have a formal solution to the $\mathcal{F}$--gradient Ricci flow along the vector field $X:=v_{1,\phi}\partial_v-\partial_u$. In other words, we have
\begin{equation*}
\begin{aligned}
(L_X g_u)_{ij}&=-2 [(\mathrm{Ric}_\phi)_u]_{ij}\bmod O(u^\infty),\\
L_X \widetilde{\phi}&=-(R_u+\Delta\phi_u)\bmod O(u^\infty).
\end{aligned}
\end{equation*} \end{theorem}
 Note that $\widetilde{g}(X,\nabla v_\phi)=0$, i.e. $X$ is orthogonal to $\nabla\phi$.

Let $\tau=\frac{1}{2\lambda}$. For $\lambda\neq 0$, we define 
\begin{equation}
\mathcal{W}_{k,\phi} :=\int_M  \tau^k v_{k,\phi} \,(4\pi\tau)^{-n/2} e^{-\phi}\mathrm{dvol}_g.
\end{equation}
When $\lambda=0$, we define 
\begin{equation}
\mathcal{F}_{k,\phi}=\int_M v_{k,\phi}e^{-\phi}\mathrm{dvol}_g.
\end{equation}
These are fully non-linear analogues of Perelman's $\mathcal{F}$--functional and $\mathcal{W}$--functional. Indeed, 
\begin{multline*}
\int_M \tau v_{1,\phi}(4\pi\tau)^{-n/2}e^{-\phi}\mathrm{dvol}_g\\=\frac{1}{2}\int_M [\tau(R+|\nabla\phi|^2)+\phi-n](4\pi\tau)^{-n/2}e^{-\phi}\mathrm{dvol}_g
\end{multline*}
coincides~\cite{Case2014s} with Perelman's $\mathcal{W}$--functional~\cite{Perelman1}. Similarly, it is easy to see that $\int_M v_{1,\phi}e^{-\phi}\mathrm{dvol}_g$ corresponds with Perelman's $\mathcal{F}$--functional.
We now define a Ricci as a gradient flow spacetime, in the spirit of the Ricci flow spacetime as defined by Kleiner-Lott~\cite{KleinerLott2017}. This is the same as a Kleiner-Lott spacetime modulo a diffeomorphism on $(M^n,g,\phi,\lambda)$ generated by $\nabla\phi$. When $\lambda=0$, we later construct a global ambient half-space from such a spacetime. When $\lambda>0$, we show that this spacetime lies inside a space that may roughly be thought of as a space in which the weighted ambient spaces form the time slices.  
\begin{definition}
\label{def-spacetime}
A \emph{gradient Ricci flow spacetime} is a tuple of the form $\left(\mathcal{M}, g_t, \phi_t, t, \partial_t,\lambda\right)$, where
\begin{enumerate}[label=(\roman*)]
\item $\mathcal{M}$ is a smooth manifold with boundary,
\item $t$ is a submersion $t: \mathcal{M} \rightarrow I$ where $I \subset \mathbb{R}$ is a time interval,
\item $\partial \mathcal{M}=$ $t^{-1}(\partial I)$
\item $\partial_t t \equiv 1$,
\item $g$ is a smooth inner product on the spatial sub-bundle $\operatorname{ker}(d t) \subset \mathcal{T} \mathcal{M}$, and we have
$$
\begin{aligned}
\mathcal{L}_{\partial_t} g & =-2 (\mathrm{Ric}+\nabla^2\phi), \\
\mathcal{L}_{\partial_t} \phi & =-(R+\Delta \phi).
\end{aligned}
$$
\end{enumerate}
\end{definition}
The function $t$ is called the ``time function". 
We now define a global weighted ambient space. Given a gradient Ricci flow spacetime, we can explicitly construct a global weighted ambient half-space, which is a pre-ambient space that satisfies the equations $\widetilde{P}_\phi=\widetilde{F}_\phi=0$ globally, i.e. every point in $M\times \mathbb{R}\times [0,\epsilon]$ for some $\epsilon>0$, and not just formally at $M\times \mathbb{R}\times \{0\}$.
\begin{definition}
    A \emph{global weighted ambient space} is a weighted pre-ambient space $(\widetilde{\mathcal{G}},\widetilde{g},\widetilde{\phi},\lambda)$ such that $\widetilde{P}_\phi=\widetilde{F}_\phi=0$.
\end{definition}
\begin{theorem}
    \label{h-embedding}
Given a gradient Ricci flow spacetime $\mathcal{M}$, there exists a global ambient half-space of the form $(\mathcal{M}\times \mathbb{R},\widetilde{g},\widetilde{\phi},\lambda)$. 
\end{theorem}
We call this a \emph{gradient Ricci flow global ambient half-space}. Note that the global ambient metric measure structure $(\widetilde{g},\widetilde{\phi})$ is not in normal form. 

A \emph{normal global ambient space} is a space that is globally a normal ambient space. In particular, it is a global pre-ambient space $(\widetilde{\mathcal{G}},\widetilde{g},\widetilde{\phi},\lambda)$ such that at any $z\in\widetilde{\mathcal{G}}$, there exist coordinates $\{v,x^i,u\}$ that satisfy the following: 
\begin{enumerate}[label=(\roman*)]
 \item  the set of $u\in\bR$ such that $(z,u)\in \cmG$ is an open interval $I_z$ containing $0$;
 \item the curve on $I_z$ defined as $u\mapsto(z,u)$ is a geodesic in $\cmG$; and
 \item it holds that $\cg=\boldsymbol{g} + 2dudv$.
\end{enumerate}
In this article, we focus on global ambient half-spaces. We prove that any global weighted ambient half-space is ambient-equivalent to a normal weighted ambient half-space. In particular, this implies that a gradient Ricci flow spacetime, and therefore a Kleiner-Lott spacetime, can be uniquely embedded in a normal weighted ambient half-space.

Note that two global weighted ambient half-spaces are said to be \emph{ambient-equivalent} if they are diffeomorphism equivalent, and the diffeomorphism $\psi$ restricts to the identity map on the conformal cone $(M\times\mathbb{R},g,\phi,\lambda)$ at $t=0$.
\begin{theorem}
\label{ricci-equivalent-normal}
Any global ambient half-space is ambient-equivalent to the normal global ambient half-space.  \end{theorem}
In particular, we deduce from \cref{ricci-equivalent-normal} that the gradient Ricci flow global ambient half-space is ambient-equivalent to a normal global ambient half-space. This allows us to conclude the following: 
\begin{theorem}
\label{every-global-ambient}
Every global ambient half-space contains a gradient Ricci flow spacetime of the form of \cref{def-spacetime}.
\end{theorem}

We now prove the monotonicity of the $\mathcal{F}_{k,\phi}$ functionals for Einstein manifolds. 
\begin{theorem}
\label{einstein-long-term-behavior}
Given a compact Einstein manifold $(M^n,g,\phi,\lambda=0)$ of the form of \cref{einstein-manifold-equation} with $\mu> 0$, the behavior of $\mathcal{F}_{k,\phi}$
for $n\geq 3$ and $j\geq 1$ is the following: 
\begin{enumerate}[label=(\roman*)]
\item if $4(j-1)<n<4j$, then $\mathcal{F}_{k,\phi}$ is strictly monotone increasing for $k\in\{1,\dots,j\}\cup\{j+2s,s\geq 1\}$, and strictly monotone decreasing for $k\in\{j+2s-1,s\geq 1\}$; and

\item if $n=4j$, then $\mathcal{F}_{k,\phi}$ is strictly monotone increasing for $k\in\{1,\dots,j\}$, and remains $0$ for $k>j$.
\end{enumerate}
\end{theorem}
By \cref{einstein-volume-coefficients}, the long-time behavior of $\mathcal{F}_{k,\phi} $ as stated in  \cref{einstein-long-term-behavior} corresponds to the sign of $v_{k,\phi}$.
\begin{enumerate}[label=(\roman*)]
\item $\mathcal{F}_{k,\phi}$ is strictly monotone increasing under the $\mathcal{F}$--gradient Ricci flow if $v_{k,\phi}> 0$;

\item $\mathcal{F}_{k,\phi}$ is strictly monotone decreasing under the $\mathcal{F}$--gradient Ricci flow if $v_{k,\phi}< 0$; and

\item $\mathcal{F}_{k,\phi}$ remains $0$ under the $\mathcal{F}$--gradient Ricci flow if $v_{k,\phi}=0$.
\end{enumerate}

We also prove the monotonicity of the $\mathcal{F}_{k,\phi}$ functional under slightly different conditions. See \cref{lemma-shifted-cone} for more details. 

We now show that a solution to the $\mathcal{W}$--gradient flow exists in a space that may roughly be thought of as a space in which weighted ambient spaces form the time slices. 
\begin{theorem}
\label{w-theorem}
Consider the space $(\widetilde{\mathcal{G}},\widetilde{g}(s),\widetilde{\phi}(s),(2\tau(s))^{-1})\times (-\epsilon,\tau)$, with $\{v,x^i,u(s)\}$ as the coordinates on $(\widetilde{\mathcal{G}},\widetilde{g}(s),\widetilde{\phi}(s),(2\tau(s))^{-1})$ and $s$ the coordinate on $(-\epsilon,\tau)$. Here $(\widetilde{\mathcal{G}},\widetilde{g}(s),\widetilde{\phi}(s),(2\tau(s))^{-1})$ is a normal global weighted ambient space for each $s$, where $(g_{u(s)}(s),\phi_{u(s)}(s),\tau(s))$ are of the form
\begin{equation*}
\begin{aligned}
 \tau(s)&=(1-s/\tau)\tau,\\
 g_{u(s)}(s)&=(1-s/\tau)g_u,\\
 \phi_{u(s)}(s)&=\phi_u.
\end{aligned}
\end{equation*}
Then along the vector field
\begin{equation*}
Y:=(1-(2\tau)^{-1} u)^{-1}(v_{1,\phi})_u\partial_v-(1-(2\tau)^{-1} u)\partial_u+\partial_s,
\end{equation*}
we have,
\begin{equation*}
\begin{aligned}
 L_Y g_{ij}&=-2(\mathrm{Ric}+\nabla^2\phi),\\
 L_Y \phi&=-R-\Delta\phi+\frac{n}{2\tau},\\
 L_Y \tau&=-1.
\end{aligned}
\end{equation*}
\end{theorem}
Using this construction, it becomes easy to deduce the monotonicity of Perelman's $\mathcal{W}$ functional, and also study the monotonicity of the higher $\mathcal{W}_{k,\phi}$ functionals. We prove that the $\mathcal{W}_{k,\phi}$ functionals are constant along the Ricci flow if and only if the manifold with density is a shrinking soliton.
\begin{theorem}
\label{einstein-w-flow}
Given a manifold with density $(M^n,g,\phi,(2\tau)^{-1})$ with $\tau>0$, the functionals $\mathcal{W}_{k,\phi}$ are constant along the Ricci flow for all $k\in\mathbb{N}$ if and only if $(M^n,g,\phi,(2\tau)^{-1})$ is a shrinking soliton such that $P_\phi=0$. 
\end{theorem}

We now try to deduce simple homogeneous linear representations of the tuple $(\partial_u^k g_u,\partial_u^k \phi_u)$. First, we define an analogue of the weighted extended obstruction tensor~\cites{Graham2009}. Let $\widetilde{R}$ denote the ambient curvature tensor, and let $k\geq 1$. We define the tensor $(\Omega_\phi)^{(k)}_{ij}$ as \begin{equation*}(\Omega_\phi)^{(k)}_{ij}:= \widetilde{R}_{\infty ij\infty,\underbrace{\scriptstyle \infty\dots\infty}_{k-1}}\Bigr|_{u=0}, \end{equation*} where $\infty$ refers to the $u$ coordinate, and $i,j$ refer to the coordinates on the manifold $M$.
\cref{examples-theorem} implies that $(\Omega_\phi)^{(k)}_{ij}=0$ when $\man$ is flat or is a gradient Ricci soliton. We now show that $ v_{k,\phi} $ can be written as homogeneous polynomials in $P_\phi$, $Y_{\phi}$ and $(\Omega_\phi)_{ij}^{(k)}$. Indeed, $\partial^k_{u}g_{ij},\partial^k_{u}\phi$ and $\frac{1}{2}g^{ij}\partial^k_{u}g_{ij}-\partial^k_{u}\phi$ can also be written in this manner (compare with \cite[Theorem 1.2]{Graham2009}). 
\begin{theorem}
\label{linear-expansion-theorem}
Let $k\geq 1$.
There exist linear combinations $\mathcal{Q}$ and $\mathcal{S}$ of partial contractions of tensor products of $Y_{\phi}, P_\phi$ and $(\Omega_\phi)_{ij}^{(k)}$ with respect to $(g_u,f_u)$ such that
\begin{equation}
\begin{aligned}
\label{linear-formulas}
\partial^k_{u}g_{ij}&=\mathcal{Q}(P_\phi,(\Omega_\phi)^{(1)}_{ij},\dots,(\Omega_\phi)^{(k-1)}_{ij}),\\
\partial^k_{u}\phi&=\mathcal{S}(Y_{\phi},P_\phi,(\Omega_\phi)^{(1)}_{ij},\dots,(\Omega_\phi)^{(k-1)}_{ij}).
\end{aligned}
\end{equation}
Similarly, there exist linear combinations $\mathcal{T}$ and $\mathcal{V}_k$ of complete contractions of tensor products of $Y_{\phi}, P_\phi$ and $(\Omega_\phi)^{(k)}$ with respect to $(g_u,f_u)$ such that
\begin{equation}
\begin{aligned}
\label{v-equation}
\frac{1}{2}g^{ij}\partial_u^{k} g_{ij}-\partial_u^{k} \phi&=\mathcal{T}(Y_{\phi},P_\phi,(\Omega_\phi)^{(1)}_{ij},\dots,(\Omega_\phi)^{(k-2)}_{ij}),\\
v_{k,\phi} &=\mathcal{V}_{k}(Y_{\phi}, P_\phi, (\Omega_\phi)^{(1)}, \ldots, (\Omega_\phi)^{(k-2)}).
\end{aligned}
\end{equation}
\end{theorem}


Given a weighted conformal class of the form of \cref{conformal-class}, we define $\mathcal{C}_1$ as \[\mathcal{C}_1:=\{(g,\phi,\tau)\mid \int_M (4\pi\tau)^{-\frac{n}{2}}e^{-\phi}\mathrm{dvol}_g=1\}.\] We consider $\tau$ to be a constant here, although we may also study $\tau$ as varying. See \cites{Case2016v,Case2014sd} for more details. We first write down the variational formula for $ v_{k,\phi} $.
\begin{theorem}
\label{infinitesimal-conformal-change-theorem}
For $\phi\in \mathcal{C}_1$, let $\{\phi_t\}_{t\in (-\epsilon,\epsilon)}\subset \mathcal{C}_1$ be a smooth variation of $\phi$ such that $\phi_t=\phi+t\omega$. For $k\geq 1$, the infinitesimal conformal variation of $v_{k,\phi}$ is
\begin{equation}
\label{v-conformal-variation}
\restr{\frac{\partial}{\partial t}}{t=0} v_{k,\phi}=\omega\lambda v_{k-1,\phi} + {\nabla^*_i}^{(0)}[(L_{k,\phi})^{ij}\nabla_j \omega],
\end{equation} where $(L_{k,\phi})^{ij}=\frac{1}{k!}\restr{\partial^k_u}{u=0}[v_{\phi}\int_0^u g^{ij}(u)\,du]$, and $-\nabla^*$ is the adjoint of $\nabla$ with respect to the $L^2$-inner product induced by the weighted volume element~\cite{Case2014s}.
\end{theorem}

\cref{v-conformal-variation} allows us to study the critical points of $\mathcal{W}_{k,\phi}$. Note that Case~\cites{Case2014sd} studies the critical points of $\mathcal{W}_{k,\phi}$ when $k\in\{1,2\}$, or when the manifold is locally conformally flat. 
\begin{theorem}
\label{first-derivative-fphim}
The critical points of $ \mathcal{W}_{k,\phi} $ in $\mathcal{C}_1$ are those for which 
\begin{equation}
\label{v-delta}
v_{k,\phi}-\lambda v_{k-1,\phi}=a
\end{equation}
for some $a\in\mathbb{R}$, where $v_{0,\phi}:=1$.
\end{theorem}
\cref{v-delta} implies that shrinking gradient Ricci solitons are critical points of $\mathcal{W}_{k,\phi}$ in $\mathcal{C}_1$ for $k\in\mathbb{N}$, thus extending~\cite[Theorems 1.3 and 9.2]{Case2014sd}.
\begin{theorem}
\label{second-derivative-theorem}
Let $\manlam$ be a compact, connected shrinking gradient Ricci soliton. Then
\[(-1)^k\restr{(\mathcal{W}_{k,\phi})^{\prime\prime}}{\mathcal{C}_1}\leq 0.\]
\end{theorem}
We can also study $(\mathcal{W}_{k,\phi})^{\prime\prime}$ for compact solitons that are steady ($\lambda=0)$, thereby recovering the relevant results from~\cites{Case2014sd}, or expanding $(\lambda<0)$, thereby recovering the relevant results from \cite{feldman-ilmanen-knopf}. Note that compact steady and expanding solitons are Einstein manifolds.  Haslhofer~\cites{HaslhoferMuller2010} has also studied the total functional for non-compact manifolds.  

We use the ambient metric for $\lambda=0$ to construct conformally invariant weighted GJMS operators. The space of conformal densities of $\cmG$ of weight $w$, denoted as $\mathcal{E}[w]$, is defined as \[\mathcal{E}[w]=\{f:\mathcal{C}\to C^\infty(M)\mid f(\phi+v)(x)=e^{wv}f(\phi)(x)\}.\]
\begin{theorem}
\label{gjmstheorem}
For every $k\in\mathbb{N}$, there is a formally self-adjoint conformally invariant operator \[L_{2k,\phi}=\mathcal{E}[-1/2]\to \mathcal{E}[-1/2]\] with leading term $\Delta_\phi^k$.
\end{theorem}
This article is organized as follows: in \cref{sec:manifolds-with-density}, we recall some weighted invariants of manifolds with density. In \cref{sec:ambient-metric}, we define ambient spaces and discuss ambient-equivalence. In \cref{sec:formal-theory}, we prove \cref{ambient-metric-theorem}. In \cref{sec:ricci}, we prove \cref{ricci-flow,h-embedding,einstein-long-term-behavior,ricci-equivalent-normal}. In \cref{sec:w-funct}, we prove \cref{w-theorem,einstein-w-flow}. In \cref{sec:gjms}, we prove \cref{gjmstheorem}. In \cref{sec:examples}, we provide explicit examples of the ambient space associated with ``simple" manifolds with density, like Einstein manifolds, gradient Ricci solitons or locally conformally flat manifolds, thereby proving \cref{examples-theorem,second-examples-theorem}. In \cref{sec:renormalized-volume-coefficients}, we prove \cref{linear-expansion-theorem,first-derivative-fphim,second-derivative-theorem,infinitesimal-conformal-change-theorem}.

\subsection*{Acknowledgements}
I would like to thank Professor Jeffrey S. Case for suggesting the problem, and for his generosity with his time and ideas. I would also like to thank the anonymous referee for their helpful suggestions and corrections, which have substantially improved the presentation of the paper. 

\section{Background}
\label{sec:manifolds-with-density}
In this section, we recall the definition of a manifold with density~\cites{Case2014sd}, and then discuss some weighted invariants.
\begin{definition}
A \emph{manifold with density} is a tuple $\man$, consisting of a Riemannian manifold $(M^n,g)$, a smooth measure $e^{-\phi}\mathrm{dvol}_g$ determined by the function $\phi\in C^\infty(M)$ and the Riemannian volume element $\mathrm{dvol}_g$, and a parameter $\lambda\in\mathbb{R}$. The tuple $(g,\phi)$ is referred to as a \emph{metric measure structure} on the manifold with density.
\end{definition}
The \emph{weighted Bianchi identity} for $\man$ is 
\begin{equation}
\label{weighted bianchi}
2(\delta_\phi \text{Ric}_\phi)-\nabla R_{\phi}=0.
\end{equation}
The \emph{weighted Bach tensor} $B_{\phi}$ is defined as~\cite[Equation (1.1)]{Case2016v}
\begin{equation}
\label{weighted bach} 
\begin{split}
B_{\phi}=&\delta_\phi d P_\phi+\mathrm{Rm}\cdot P_\phi,
\end{split}
\end{equation}
where
\begin{align*}
\left(\delta_\phi d P_\phi\right)(X, Y)&:=\sum_{j=1}^n \nabla_{E_i} P_\phi\left(E_i, X, Y\right)-d P_\phi(\nabla \phi, X, Y),\\
\left(\operatorname{Rm} \cdot P_\phi\right)(X, Y)&:=\sum_{j=1}^n \operatorname{Rm}\left(E_i, X, E_j, Y\right){P_\phi}\left(E_i, E_j\right)
\end{align*}
for all $p \in M$ and all $X, Y \in T_p M$, where $\left\{E_i\right\} \subset T_p M$ is an orthonormal basis.

\begin{definition}
\label{conformalchangedefinition}
Two metric measure structures $(g,\phi)$ and $(\widehat{g},\widehat{\phi})$ are defined to be \emph{pointwise conformally equivalent} if $(\widehat{g},\widehat{\phi})=(g,\phi+\omega)$ for some $\omega\in C^\infty(M)$.
\end{definition}
We have the following conformal transformation formulas~\cite[Lemma (3.3)]{Case2014s}.
\begin{lemma}
Let $\manlam$ be a manifold with density, and let $\widehat{\phi}=\phi+\omega$. Then
\begin{equation}
\label{conformaltransformationformulas}
\begin{aligned}
\widehat{R}_{\phi}&=R_{\phi}+2\Delta\omega-2g(\nabla\phi,\nabla\omega)-|\nabla\omega|^2+2\lambda \omega,\\
\widehat{F}_\phi&=F_\phi+\Delta\omega-2g(\nabla\phi,\nabla\omega)-|\nabla\omega|^2+2\lambda\omega,\\
\widehat{\mathrm{Ric}}_{\phi}&=\mathrm{Ric}_{\phi}+\nabla^2\omega,\\
\widehat{Y}_{\phi}&=Y_{\phi}-\frac{1}{2}[2g(\nabla\phi,\nabla\omega)+|\nabla\omega|^2-2\lambda\omega].
\end{aligned}    
\end{equation}
\end{lemma}
\begin{definition}
\label{lcf-definition}
A manifold with density $\man$ is weighted \emph{locally conformally flat}~\cite{Case2014sd} if $g$ is a flat metric.
\end{definition}
\begin{definition}
\label{grs-definition}
\emph{Gradient Ricci solitons} are manifolds with density, denoted as $\manlam$, such that
\begin{equation}
\begin{aligned}
\label{soliton-def}
P_\phi&=0.   
\end{aligned}    
\end{equation}
\end{definition}
Note that if $g$ is flat and $\lambda\neq 0$, there need not exist a $\phi$ such that $\manlam$ is a gradient Ricci soliton.
\begin{lemma}
For a gradient Ricci soliton, we have $F_\phi=c$ for some $c\in\mathbb{R}$.
\end{lemma}
\begin{proof}
We obtain from \cref{soliton-def} and the Bianchi identity that 
\begin{equation}
\label{2nd-identity}
|\nabla\phi|^2+R-2\lambda \phi=c
\end{equation}
for some $c\in \mathbb{R}$. Subtracting \cref{2nd-identity} from the trace of \cref{soliton-def} implies that $F_\phi$ is constant.
\end{proof}
\section{Weighted ambient metric}
\label{sec:ambient-metric}
In this section, we define the weighted ambient space and ambient-equivalence.

Let $\man$ be a manifold with density. Denote by $\mathcal{E}$ the trivial line bundle over $M$. Let $\mG$ be the ray subbundle of $\mathrm{S}^2(T^*M)\oplus \mathcal{E}$ consisting of all triples $(x,g_x,l)$ such that $l=\phi(x)+v$ for some $v\in \mathbb{R}$ and $x\in M$. We define the dilation $\delta_s\colon \mG \mapsto \mG$ by $(x,g_x,l)\mapsto (x,g_x,l+s)$. Let $T:=\restr{\frac{d}{ds}\delta_s}{{s=0}}$ denote the infinitesimal generator of dilations. Also, let $\pi \colon (x,g_x,l) \mapsto x$ be the projection from $\mG$ to $M$. There is a canonical metric measure structure $(\boldsymbol{g},\boldsymbol{\phi})$ on $\mG$ defined by $\boldsymbol{g}(X,Y)=g_x(\pi_*X,\pi_*Y)$ and $\boldsymbol{\phi}(x,g_x,l)=l$.

Fixing $(g,\phi)$, we define a coordinate chart \[\mG\mapsto \mathbb{R}\times M:(x,g_x,\phi(x)+v)\mapsto (-v,x)\] for $v\in\mathbb{R}$. Moreover, we have the embedding \[\iota \colon \mG\ \hookrightarrow \mG\times\bR,(-v,x)\mapsto (-v,x,0).\]
Note that if $(x^1,\dots x^n)$ denote the local coordinates on $M$, then we consider $(v,x^1,\dots,x^n,u)$ to be the local coordinates on $\mG\times \bR_+$. We use $0$ to label the $v$-component, $\infty$ to label the $u$-component, lower case Latin to label coordinates on $M$ and upper case Latin to label coordinates on $\mathcal{G}\times \bR_+$. We choose the sign convention $(-v,x,0)$ instead of $(v,x,0)$ for the embedding, in order to ensure consistency with \cites{CaseKhaitan2022,FeffermanGraham2012}. For example, if the embedding was of the form $(v,x,0)$, then $\restr{(\partial_u g,\partial_u \phi)}{u=0}$ would be $(-2P_\phi,Y_\phi)$, which would be inconsistent with the formulas $\restr{(\partial_\rho g,\partial_\rho \phi)}{\rho=0}=(2P_\phi,-Y_\phi)$ in \cite{CaseKhaitan2022}, and the formula $\restr{\partial_\rho g}{\rho=0}=2P$ in \cite{FeffermanGraham2012}. 

We extend the dilation to $\mathcal{G}\times \bR_+$ as $\delta_s (v,x,u)= (v+(\lambda u-1)^{-1}s,x,u)$. Hence, we have \begin{equation}
\label{t-equation}
T=(\lambda u-1)^{-1}\partial_v,
\end{equation}
where $T$ is the infinitesimal generator of dilations $\delta_s$. 

Given $\phi\in C^\infty(M)$, we define \[[\phi]:=\{\phi+\omega\mid \omega\in C^\infty(M)\}\] to be the conformal structure of $\phi$.
We now define pre-ambient spaces and special cases thereof.
\begin{definition}
\label{preambient-defintion}
$(\cmG,\widetilde{g},\widetilde{\phi},\lambda)$ is called a \emph{pre-ambient space}  of the conformal manifold $(M^n,g,[\phi],\lambda)$ if

\begin{enumerate}[label=(\roman*)]
 \item $\cmG$ is a dilation-invariant neighborhood of $\mG\times \{0\}$;
 \item $\cg$ is a smooth metric of signature $(n+1,1)$ and $\widetilde{\phi}$ a smooth function on $\cmG$; 
 \item  $\delta_s^*\widetilde{g}= \widetilde{g}$ and $\delta_s^*\widetilde{\phi}=\widetilde{\phi}+s$;
 \item 
 \label{iv}
$(\iota^*\widetilde{g},\iota^*\widetilde{\phi})=(\boldsymbol{g},\boldsymbol{\phi})$; and 
 \item \label{v}
 \label{straight-definition}
 for each $p\in \widetilde{\mathcal{G}}$, the dilation orbit $s\mapsto \delta_s p$ is formally a geodesic for $\widetilde{g}$ to first order, i.e. $\widetilde{\nabla}_{\partial_v}\partial_v=0 \bmod \,O(u)$.
\end{enumerate}
\end{definition}
 By the homogeneity requirements of $(\widetilde{g},\widetilde{\phi})$, $\widetilde{g}$ is independent of $v$, and $\widetilde{\phi}$ is of the form $\phi_u-v+\lambda uv$, where $\phi_u$ is independent of $v$ and $\restr{\phi_u}{u=0}=\phi$.
 
\cref{preambient-defintion} is quite similar to the standard definition of the pre-ambient metric~(cf.\ \cite{FeffermanGraham2012}*{Definition 2.7}, \cite{CaseKhaitan2022}*{Definition 3.1}) except for \cref{straight-definition}, which is generally not included as a part of the definition of a pre-ambient metric. In fact, \cref{preambient-defintion} is actually the definition of a pre-ambient metric that is \emph{straight}~(cf.\ \cite{FeffermanGraham2012}) to first order. We introduce \cref{straight-definition} to remove the indeterminacy in the values of $\widetilde{g}_{00},\widetilde{g}_{0i}$ in a normal ambient metric. Our focus in the current paper is on a specific class of metrics for which we observe gradient Ricci flow; i.e. we have 
\begin{equation}
\label{ricci-flow-intro}
\begin{aligned}
(L_X g_u)_{ij}&=-2 [(\mathrm{Ric}_\phi)_u]_{ij}\bmod O(u^\infty),\\
L_X \widetilde{\phi}&=-(R_u+\Delta\phi_u)\bmod O(u^\infty)
\end{aligned}
\end{equation} for some vector field $X\in \mathfrak{X}(M)$, and ambient metrics that satisfy \cref{straight-definition} suffice for our purposes.

Also note that \cref{straight-definition} is no longer equivalent to $\nabla T=\mathrm{Id}$, or that $\widetilde{g}(2 T, \cdot)=d(\widetilde{g}(T, T))$, as it is in~\cite[Proposition (2.4)]{FeffermanGraham2002}. This is clear from \cref{t-equation}.

We now define \emph{normal} pre-ambient spaces.
\begin{definition}
\label{normal-definition}
A pre-ambient space $\amb$ is said to be in \emph{normal form} relative to $(g,\phi)$ if

\begin{enumerate}[label=(\roman*)]
 \item  for each fixed $z\in\mG$, the set of $u\in\bR$ such that $(z,u)\in \cmG$ is an open interval $I_z$ containing $0$;
 \item for each $z\in\mG$, the curve on $I_z$ defined as $u\mapsto(z,u)$ is a geodesic in $\cmG$; and
 \item 
 \label{normal-form}
 on $\mG$, it holds that $\cg=\boldsymbol{g} + 2dudv$.
\end{enumerate}
\end{definition}
\cref{normal-definition} is identical to the definition of the normal ambient metric in~\cites{FeffermanGraham2012,CaseKhaitan2022}.

We now show that for a normal ambient metric, $\cg_{0\infty}=1$, and $\cg_{\infty i}=\cg_{\infty\infty}=0$. 
\begin{lemma}
\label{straight-metric-lemma}
Let $\amb$ be a pre-ambient space such that for each $z\in\mG$, the set of all $u\in\bR$ such that $(z,u)\in\cmG$ is an open interval containing $0$. Then $(\cg,\cphi)$ is in \textit{normal} form if and only if $\cg_{0\infty}=1$ and $\cg_{\infty i}=\cg_{\infty\infty}=0$.
\end{lemma}
We recall that in~\cites{CaseKhaitan2022,FeffermanGraham2012}, a normal ambient metric satisfies the conditions $\cg_{0\infty}=t$, and $\cg_{\infty i}=\cg_{\infty\infty}=0$. Here $t$ corresponds to the $0$ coordinate.
\begin{proof}
A normal pre-ambient space has 
\begin{equation*}
\begin{aligned}
\restr{\widetilde{g}_{\infty\infty}}{\mathcal{G}}=\restr{\widetilde{g}_{i\infty}}{\mathcal{G}}&=0,\\ \restr{\widetilde{g}_{0\infty}}{\mathcal{G}}&=1.
\end{aligned}
\end{equation*}
The $u$-lines are geodesics if and only if the Christoffel symbols $\widetilde{\Gamma}_{\infty \infty I}$, $I\in \{0,i,\infty\}$, vanish. Taking $I=\infty$ gives $\partial_{u} \widetilde{g}_{\infty \infty}=0$, and hence $\widetilde{g}_{\infty \infty}=0 .$ Now taking $I=i$ and $I=0$ gives $\widetilde{g}_{\infty i}=0$ and $\widetilde{g}_{\infty 0}=1$.  

The converse follows similarly.
\end{proof}
Finally, we define a weighted ambient space.
\begin{definition}
\label{ambient-definition-mfd}
A weighted ambient space for the conformal manifold $(M^n,g,[\phi],\lambda)$ is a weighted pre-ambient space $\amb$ such that at $u=0$, we have
\begin{equation}
\begin{aligned}
\label{ambient-condition}
\widetilde{P}_\phi&\equiv O(u^\infty),\\
\widetilde{F}_\phi&\equiv O(u^\infty).
\end{aligned}    
\end{equation}
\end{definition}
Now we discuss the equivalence of ambient metrics.
\begin{definition}
\label{ambient-equivalent-definition}
We say that two pre-ambient spaces $(\widetilde{\mathcal{G}}_1,\widetilde{g}_1,\widetilde{\phi}_1,\lambda)$ and $(\widetilde{\mathcal{G}}_2,\widetilde{g}_2,\widetilde{\phi}_2,\lambda)$ for $\man$ are \emph{ambient-equivalent} if there exist open sets $\mathcal{U}_{1} \subset \widetilde{\mathcal{G}}_{1}$ and $\mathcal{U}_{2} \subset \widetilde{\mathcal{G}}_{2}$, and a diffeomorphism $\psi: \mathcal{U}_{1} \rightarrow \mathcal{U}_{2}$ such that
\begin{enumerate}[label=(\roman*)]
\item $\mathcal{U}_{1}$ and $\mathcal{U}_{2}$ both contain $\mathcal{G} \times\{0\}$;

\item $\mathcal{U}_{1}$ and $\mathcal{U}_{2}$ are dilation-invariant and $\psi$ commutes with dilations;

\item the restriction of $\psi$ to $\mathcal{G} \times\{0\}$ is the identity map; and

\item $(\widetilde{g}_{1}-\psi^{*} \widetilde{g}_{2}, \widetilde{\phi}_1-\psi^*\widetilde{\phi}_2)\equiv O(u^\infty)$;
\end{enumerate}
\end{definition} 

Recall that $\phi$ determines the fiber coordinate $v:\mG\to \mathbb{R}$ and the horizontal subbundle $\mathrm{ker}(dv)\subset T\mathcal{G}$. 

\cref{ambient-equivalent-definition} is identical to~\cite{CaseKhaitan2022}*{Definition 3.9}, except that we have equivalence to infinite order. One may think of this being a consequence of the fact that $m=\infty$. 

A key step in the proof of \cref{ambient-metric-theorem} is the construction and classification of normal weighted ambient metrics~(cf.\ \cite{FeffermanGraham2012}). 

\begin{theorem}
\label{normal-ambient-metric-exists}
Given $(M^n,g,[\phi],\lambda)$, $n\geq 3$, there exists a normal weighted ambient space $(\widetilde{\mathcal{G}},\widetilde{g},\widetilde{\phi},\mu)$ for it.  Moreover, if $(\widetilde{\mathcal{G}}_j,\widetilde{g}_j,\widetilde{\phi}_j,\mu)$, $j\in\{1,2\}$, are two such weighted normal ambient spaces, then $(\widetilde{g}_1-\widetilde{g}_2,\widetilde{\phi}_1-\widetilde{\phi}_2)\equiv O(u^\infty)$.
\end{theorem} We prove \cref{normal-ambient-metric-exists} in \cref{sec:formal-theory}.
We now show that any ambient space is ambient-equivalent to a normal ambient space. In order to do so, we first construct a $g$-transversal vector $V_z$. Recall that $\phi$ determines the fiber coordinate $s: \mathcal{G}\to \mathbb{R}$ and the horizontal subbundle $\mathcal{M}=\mathrm{ker} (ds)\subset T\mathcal{G}$. For $z\in\mathcal{G}$, $\mathcal{M}_z$ may be viewed as a subspace of $T_{(z,0)}(\mathcal{G}\times \mathbb{R})$ via the inclusion $\iota:\mathcal{G}\to \mathcal{G}\times \mathbb{R}$. We closely follow the proof of~\cite[Proposition (2.8)]{FeffermanGraham2012}.
\begin{lemma}
\label{transversal-lemma}
There exists a vector $V_z \in T_{(z, 0)}(\mathcal{G} \times \mathbb{R})$ such that 
\begin{equation}
\label{Vconditions}
\begin{aligned}
\widetilde{g}(V,T)&=1,\\
\widetilde{g}(V,X)&=0\text{ for all }X\in\mathcal{M}_z,\\
\widetilde{g}(V,V)&=0.
\end{aligned}
\end{equation}
Moreover, $V_{z}$ is transverse to $\mathcal{G} \times\{0\}, V_{z}$ depends smoothly on $z$, and $V_{z}$ is dilation-invariant in the sense that $\left(\delta_{s}\right)_{*} V_{z}=V_{\delta_{s}(z)}$ for $s \in \mathbb{R}_{+}$and $z \in \mathcal{G}$. 
\end{lemma}

\begin{proof}
Let 
\begin{equation*}
\begin{aligned}
V&=V^0\partial_v+V^i\partial_{x^i}+V^\infty\partial_u,\\
\widetilde{g}&=\widetilde{g}_{ij}(x)dx^i dx^j+2\widetilde{g}_{0\infty}dvdu+2\widetilde{g}_{i\infty}dx^i du+\widetilde{g}_{\infty\infty}(du)^2.
\end{aligned}    
\end{equation*}
\cref{Vconditions} becomes
\begin{equation}
\label{Vequationsproof}
\begin{aligned}
-\widetilde{g}_{0\infty}V^0&=1,\\
\widetilde{g}_{ij}V^j+\widetilde{g}_{i\infty}V^\infty&=0\\
2 \widetilde{g}_{0 \infty} V^{0} V^{\infty}+ \widetilde{g}_{i j} V^{i} V^{j}+2 \widetilde{g}_{i \infty} V^{i} V^{\infty}+\widetilde{g}_{\infty \infty}\left(V^{\infty}\right)^{2}&=0.
\end{aligned}    
\end{equation}
We can use \cref{Vequationsproof} to solve uniquely for $V^0,V^i$ and $V^\infty$. The dilation invariance and smoothness follow from the properties defining a pre-ambient space. 
\end{proof}
We now prove that the a weighted ambient space for a conformal manifold with density is ambient equivalent to a normal ambient space.
\begin{lemma}
\label{pullback-normal}
Let $\amb$ be a weighted ambient space for the conformal manifold with density $(M^n,g,[\phi],\lambda)$. Then it is ambient-equivalent to a normal ambient space for $(M^n,g,[\phi],\lambda)$.
\end{lemma}
\begin{proof}
For $z \in \mathcal{G}$, let $V_z$ be the $g$-transversal for $\widetilde{g}$ at $(z, 0)$. Let $t\mapsto \psi(z,t) \in \widetilde{\mathcal{G}}$ be a geodesic for $\widetilde{g}$, with initial conditions
\begin{equation}
\label{initial-conditions}
\psi(z, 0)=\left.(z, 0), \quad \partial_t \psi(z,t)\right|_{t=0}=V_z.
\end{equation}
Since $\widetilde{g}$ needn't be geodesically complete, $\psi(z,t)$ is defined only for $(z,t)$ in an open neighborhood $\mathcal{U}_0$ of $\mathcal{G} \times\{0\}$ in $\mathcal{G} \times \mathbb{R}$. Since $\widetilde{g}$ and $V_z$ are homogeneous with respect to the dilations, we may take $\mathcal{U}_0$ to be dilation-invariant. Thus, $\psi: \mathcal{U}_0 \rightarrow \widetilde{\mathcal{G}}$ is a smooth map, commuting with dilations, and satisfying \cref{initial-conditions}.

Since $V_z$ is transverse to $\mathcal{G} \times\{0\}$, it follows that $$\mathcal{U}_1=\{(z,t) \in \mathcal{U}_0: \operatorname{det} \psi^{\prime}(z,t) \neq 0\}$$ is a dilation-invariant open neighborhood of $\mathcal{G} \times\{0\}$ in $\mathcal{U}_0$. Thus, $\psi$ is a local diffeomorphism from $\mathcal{U}_1$ into $\widetilde{\mathcal{G}}$, commuting with dilations. Moreover, by definition of $\psi$, we have the following: let $z \in \mathcal{G}$ and let $I$ be an interval containing 0. If $(z,t) \in \mathcal{U}_1$ for all $t \in I$, then $I \ni t\mapsto \psi(z,t)$ is a geodesic for $\widetilde{g}$, with initial conditions \cref{initial-conditions}.

Let $z \in \mathcal{G}$ and let $I$ be an interval containing 0. Assume that $(z,t) \in \mathcal{U}_1$ for all $\lambda \in I$. Then $I \ni t\mapsto \psi(z,t)$ is a geodesic for $\widetilde{g}$, with initial conditions \cref{initial-conditions}.
The map $\psi$ need not be globally one-to-one on $\mathcal{U}_1$. However, arguing as in the proof of the Tubular Neighborhood Theorem~(cf.\ see \cite{Bredon1993TopologyAG}), one concludes that there exists a dilation-invariant open neighborhood $\mathcal{U}_2$ of $\mathcal{G} \times\{0\}$ in $\mathcal{U}_1$, such that $\left.\psi\right|_{\mathcal{U}_2}$ is globally one-to-one. Thus, $\psi$ is a diffeomorphism from $\mathcal{U}_2$ to a dilation invariant open subset of $\widetilde{\mathcal{G}}$ containing $\mathcal{G} \times\{0\}$.

Next, we define $\mathcal{U}=\left\{(z,t) \in \mathcal{U}_2:(z, \mu) \in \mathcal{U}_2\right.$ for all $\mu \in \mathbb{R}$ for which $\left.|\mu| \leq|t|\right\}$. Thus, $\mathcal{U}$ is a dilation-invariant open neighborhood of $\mathcal{G} \times\{0\}$ in $\mathcal{U}_2$. Moreover, for each fixed $z \in \mathcal{G},\{\lambda \in \mathbb{R}:(z,t) \in \mathcal{U}\}$ is an open interval $I_z$ containing 0. It follows that for each fixed $z \in \mathcal{G}$, the parametrized curve $I_z \ni t \mapsto \psi(z,t)$ is a geodesic for the metric $\widetilde{g}$.

Since $(\widetilde{\mathcal{G}}, \widetilde{g})$ is a pre-ambient space for $(M,[g])$, so is $\left(\mathcal{U}, \psi^* \widetilde{g}\right)$. For each fixed $z \in \mathcal{G}$, the parametrized curve $I_z \ni t \mapsto(z,t)$ is a geodesic for $\psi^* \widetilde{g}$. From the facts that $V$ satisfies \cref{Vconditions} and $\psi$ satisfies \cref{initial-conditions}, it follows that under the identification $\mathbb{R}_{+} \times M \times \mathbb{R} \simeq \mathcal{G} \times \mathbb{R}$ induced by $g$, we have at $t=0$:
$$
\begin{aligned}
& \left(\psi^* \widetilde{g}\right)\left(\partial_t, T\right)=1, \\
& \left(\psi^* \widetilde{g}\right)\left(\partial_t, X\right)=0 \text { for } X \in T M, \\
& \left(\psi^* \widetilde{g}\right)\left(\partial_t, \partial_t\right)=0 .
\end{aligned}
$$
Together with \cref{iv}, these equations show that $\psi^* \widetilde{g}=\mathbf{g}_0+2 d v dt$ when $t=0$. This establishes the existence part of \cref{pullback-normal}. The uniqueness follows from the fact that the above construction of $\psi$ is forced. If $\psi$ is any diffeomorphism with the required properties, then at $u=0, \psi_*\left(\partial_t\right)$ is a $g$-transversal for $\widetilde{g}$, so must be $V$. Then for $z \in \mathcal{G}$, the curve $I_z \ni t \mapsto \psi(z,t)$ must be the unique geodesic satisfying the initial conditions \cref{initial-conditions}. These requirements uniquely determine $\psi$ on $\mathcal{U}$.
\end{proof}
Clearly, our proof closely follows~\cite{FeffermanGraham2012}*{Proposition 2.8}. The two minor poitns of departure are that we have $\left(\psi^* \widetilde{g}\right)\left(\partial_t, T\right)=1$, and that our normal ambient metric is of the form $\psi^* \widetilde{g}=\mathbf{g}_0+2 d v dt$.

We now prove the uniqueness of weighted ambient metrics, assuming the existence of a weighted ambient metric from \cref{normal-ambient-metric-exists}.
\begin{lemma}
\label{any-two-ambient-metrics}
Any two weighted ambient spaces for $(M^n,g,[\phi],\lambda)$, $n\geq 3$, are ambient-equivalent. 
\end{lemma}
\begin{proof}
We pick a representative $(g,\phi)$ and invoke \cref{normal-ambient-metric-exists}. Then there exists a normal weighted ambient space $(\widetilde{\mathcal{G}}_1,\widetilde{g}_1,\widetilde{\phi}_1,\lambda)$ for the manifold with density $(M,g,\phi,\lambda)$. Now let $(\widetilde{\mathcal{G}}_2,\widetilde{g}_2,\widetilde{f}_2,\lambda)$ be a weighted ambient space for $(M,g,[\phi],\lambda)$. Applying \cref{pullback-normal}, we find that $(\widetilde{\mathcal{G}}_2, \widetilde{g}_2,\widetilde{\phi}_2,\lambda)$ is ambient-equivalent to a weighted ambient space in normal form relative to $(g,\phi)$. By \cref{normal-ambient-metric-exists}, this space is ambient-equivalent to $(\widetilde{\mathcal{G}}_1,\widetilde{g}_1,\widetilde{\phi}_1,\lambda)$.
\end{proof}
\section{Formal theory}
\label{sec:formal-theory}
In this section, we prove \cref{ambient-metric-theorem}. Our proof closely follows the construction of the ambient metric in~\cites{FeffermanGraham2012,CaseKhaitan2022}.

Given a metric measure structure $(g,\phi)$, we deduce using \cref{straight-metric-lemma} that%
\begin{equation}
\label{generalformambientmetric}
\begin{aligned}
\widetilde{g}&=a dv^2+ 2b_i dx^i dv + 2 du dv+ g_u dx^i dx^j,\\
\widetilde{\phi}&=\phi_u-v+\lambda uv.
\end{aligned}
\end{equation}

The Christoffel symbols for $\widetilde{g}$ are 
\begin{equation}
\label{christoffel symbols for the undetermined metric}
\begin{aligned}
\widetilde{\Gamma}_{IJ}^0&=\begin{pmatrix} -\frac{1}{2} a^\prime&-\frac{1}{2} b_i^\prime&0\\-\frac{1}{2} b_i^\prime&-\frac{1}{2} g^\prime_{ij}&0\\0&0&0\end{pmatrix},\\
\widetilde{\Gamma}_{IJ}^k&=\begin{pmatrix}-\frac{1}{2}g^{kl}\partial_l a&\frac{1}{2}g^{kl}(\partial_i b_l-\partial_l b_i)&\frac{1}{2}g^{kl}b_l^\prime\\\frac{1}{2}g^{kl}(\partial_i b_l-\partial_l b_i)&\Gamma_{ij}^k&\frac{1}{2}g^{kl}g^\prime_{li}\\\frac{1}{2}g^{kl}b_l^\prime&\frac{1}{2}g^{kl}g^\prime_{li}&0 \end{pmatrix},\\
\widetilde{\Gamma}_{IJ}^\infty&=\begin{pmatrix}0&\frac{1}{2} \partial_i a&\frac{1}{2} a^\prime\\\frac{1}{2} \partial_i a&\frac{1}{2} (\partial_j b_i+\partial_i b_j) &\frac{1}{2} b^\prime_i\\\frac{1}{2} a^\prime&\frac{1}{2} b_i^\prime&0 \end{pmatrix}.
\end{aligned}
\end{equation}
The rows and columns in the matrices in \cref{christoffel symbols for the undetermined metric} correspond to the coordinates $v,x^i$ and $u$ from top to bottom and left to right respectively. From \cref{christoffel symbols for the undetermined metric}, we get
\begin{align}
\begin{split}
(\widetilde{P}_\phi)_{00}&=-\frac{1}{2}\delta_\phi(\nabla a)-\frac{1}{4}a'g^{kl}(\partial_k b_l-\partial_l b_k)-\frac{1}{2}(a')^2\\
&-\frac{1}{4}g^{kp}g^{ls}(\partial_l b_p-\partial_p b_l)(\partial_k b_s-\partial_s b_k)+\frac{1}{2}a'(-1+\lambda u),
\end{split}\nonumber\\
\begin{split}
(\widetilde{P}_\phi)_{0i}&=\frac{1}{2}\partial_i a'+\frac{1}{2}\delta_\phi(\partial_i b-\partial b_i)+\frac{1}{2}g^{kl}(\partial_i b_l-\partial_l b_i)+\frac{1}{4}(\partial_i a)g^{kl}g'_{kl}\\
&-[\frac{1}{2}a'b_i'-\frac{1}{4}(g^{kl}\partial_l a)(\partial_i b_k +\partial_k b_i)-\frac{1}{4}b'_l g^{kl}(\partial_i b_k-\partial_k b_i)\\&+\frac{1}{4}g^{kl}b'_l(\partial_i b_k+\partial_k b_i)+\frac{1}{4}(\partial_l a)g^{lp}g'_{ip}]-\frac{1}{2}b_i'(1-\lambda u)-\frac{1}{2}(\partial_i a)(\phi'+\lambda v),
\end{split}\nonumber\\
\begin{split}
(\widetilde{P}_\phi)_{ij}&=-\frac{1}{2}g'_{ij}(1-\lambda u)+(P_\phi)_{ij}+\frac{1}{2}(\partial_j b_i'+\partial_i b_j'-b_i'b_j')\\
&-\frac{1}{4}g'_{ij}g^{ks}(\partial_k b_s-\partial_s b_k)+\frac{1}{4}(\partial_i b_j+\partial_j b_i)g^{kl}g'_{kl}
-\frac{1}{2}g'_{il}g^{lk}(\partial_j b_k-\partial_k b_j)\\ &+\frac{1}{2}g^{ks}g'_{is}(\partial_j b_k+\partial_k b_j)-\frac{1}{2}(\partial_j b_i+\partial_i b_j)(\phi'+\lambda v),
\end{split}\nonumber\\
\begin{split}
(\widetilde{P}_\phi)_{0\infty}&=\frac{1}{2}\delta_\phi b'+\frac{1}{2}a''-\frac{1}{4}g^{ki}g^{ls}g'_{sk}(\partial_l b_i-\partial_i b_l)+\frac{1}{4}a'(g')_l^l-\frac{1}{2}a'(\phi'+\lambda v),
\end{split}\nonumber\\
\begin{split}
(\widetilde{P}_\phi)_{i\infty}&=\frac{1}{2}(\delta_\phi g')_i+\frac{1}{2}b_i''-\frac{1}{2}\partial_i (g')_k^k+\frac{1}{4}(b_i')(g')_l^l+\partial_i \phi'-\frac{1}{2}(\phi'+\lambda v) b_i',
\end{split}\nonumber\\
\begin{split}
(\widetilde{P}_\phi)_{\infty\infty}&=-\frac{1}{2}g^{kl}g''_{kl}-\frac{1}{4}(g')^{kl}g'_{kl}+\phi''.
\end{split}\nonumber
\end{align}    
\begin{lemma}
\label{straighta=0}
If $\widetilde{g}$ is a normal pre-ambient metric, then 
\begin{equation}
\label{straight-equations}
\begin{aligned}
\widetilde{g}_{00}\equiv 0 \bmod O(u^\infty),\\
\widetilde{g}_{0i}\equiv 0 \bmod O(u^\infty).
\end{aligned}
\end{equation}
\end{lemma}
\begin{proof}
From~\cref{straight-definition}, which states that $\widetilde{\nabla}_{\partial_v} \partial_v=0\bmod O(u)$, we infer that $\widetilde{\Gamma}_{00}^I\equiv 0\bmod O(u)$. We see from \cref{christoffel symbols for the undetermined metric} that this is equivalent to $\widetilde{g}_{00}\equiv 0 \bmod O(u^2)$. Using \cref{normal-form,christoffel symbols for the undetermined metric}, we deduce that at $\mathcal{G}$,
\begin{subequations}
\label{ricciequations}
\begin{align}
\begin{split}
\label{00component}
(\widetilde{P}_\phi)_{00}&\equiv -\left[\frac{1}{2} a^\prime(1+ a^\prime)\right]\bmod O(u),
\end{split}\\
\begin{split}
\label{0icomponent}
(\widetilde{P}_\phi)_{0i}&\equiv \left[\frac{1}{2} (\partial_i a^\prime)-\frac{1}{2} (a^\prime)(b_i^\prime)-\frac{1}{2} b_i^\prime\right]\bmod O(u).
\end{split}
\end{align}
\end{subequations}

\label{000ivanishing}
From \cref{00component}, we conclude that $(\widetilde{P}_\phi)_{00}\equiv O(u)$. Moreover, we conclude from \cref{00component,0icomponent} that if $(\widetilde{P}_\phi)_{0i}\equiv O(u)$, then $\widetilde{g}_{0i}\equiv O(u^2)$. We deduce from the expression for $(\widetilde{P}_\phi)_{00}$ and $\widetilde{g}_{00}\equiv 0\bmod \, O(u^2)$ that if $(\widetilde{P}_\phi)_{00}\equiv 0\bmod \,O(u^\infty)$, then $\widetilde{g}_{00}\equiv 0\bmod\, O(u^\infty)$.

Let us now assume that there exist $(\widetilde{g}_{0i}^{(m-1)},\widetilde{\phi}^{(m-1)})$ such that
\begin{equation}
\label{000i-iterative-first-step}
\begin{aligned}
\widetilde{g}^{(m-1)}_{0i}&=0,\\
(\widetilde{P}_\phi)^{(m-1)}_{00}&\equiv O(u^{m-1}),\\(\widetilde{P}_\phi)^{(m-1)}_{0i}&\equiv O(u^{m-1}),
\end{aligned}
\end{equation} for some $m\geq 2$. We define \[\widetilde{g}^{(m)}_{0i}=\psi_{0i}u^m,\quad \widetilde{\phi}^{(m)}=\phi^{(m-1)}+u^m\phi^{(m)}, \] 
and compute using \cref{straight-equations,000i-iterative-first-step} that 
\begin{equation*}
\begin{aligned}
(\widetilde{P}_\phi)_{00}&\equiv  O(u^m)\\
(\widetilde{P}_\phi)_{0i}&\equiv -\frac{m}{2}\psi_{0i}\bmod O(u^m).
\end{aligned}    
\end{equation*}
Hence, for all $m\geq 2$, we can choose $\psi_{0i}=0$ such that $(\widetilde{P}_\phi)_{00}=(\widetilde{P}_\phi)_{0i}\equiv O(u^m)$. This implies that $(\widetilde{P}_\phi)_{00}=(\widetilde{P}_\phi)_{0i}\equiv O(u^\infty)$ if $\widetilde{g}_{00}=\widetilde{g}_{0i}\equiv O(u^\infty)$.
\end{proof}
Note that it is because of \cref{straight-definition} that we deduce that $a'=0$. We see from \cref{0icomponent} that $a'\equiv -1+O(u)$ is also possible for a normal weighted ambient metric that does not satisfy \cref{straight-definition}. 
\begin{proof}[Proof of \cref{normal-ambient-metric-exists}]
Assuming that $\widetilde{g}_{00}=\widetilde{g}_{0i}\equiv O(u^\infty)$, we now iteratively construct a normal weighted ambient space.
Let $m\in\mathbb{N}$ be such that there is a metric
\[ \widetilde{g}_{IJ}^{(m-1)} =  \begin{pmatrix}
0  & 0   & 1 \\
0 & g_{u} & 0 \\
1         & 0            & 0
\end{pmatrix}\]
and a function $\widetilde{\phi}^{(m-1)}=\phi_u-v+\lambda uv$, such that 
\begin{align*}
\widetilde{P}_\phi^{(m-1)}&\equiv O(u^{m-1}),\\
(\widetilde{F}_\phi)^{(m-1)}&\equiv O(u^{m-1}).
\end{align*}
Note that the existence of $\widetilde{g}_{IJ}^{(1)}$ and $\widetilde{\phi}^{(1)}$ trivially holds. Now let $\cg^{(m)}_{IJ}=\cg^{(m-1)}_{IJ}+\Phi_{IJ}$ and $\widetilde{\phi}^{(m)}= \widetilde{\phi}^{(m-1)}+u^{m}\phi^{(m)}$, where 
\begin{align*}
\Phi_{IJ}&= u^m \begin{pmatrix}
0  & 0   & 0 \\
0 & \psi_{ij} & 0 \\
0          & 0            & 0
\end{pmatrix},
\end{align*} and $\psi_{ij},\phi^{(m)}$ depend only on $x$. 
We seek $\psi_{IJ}$ and $\upsilon$ such that 
 \begin{align*}
\widetilde{P}_\phi^{(m)}&\equiv  O(u^{m}),\\
(\widetilde{F}_\phi)^{(m)}&\equiv   O(u^{m}).
\end{align*}
The Christoffel symbols for $\widetilde{g}^{(m)}$ are
\begin{equation}
\label{christoffel symbols for the iterative formula}
\begin{aligned}
\widetilde{\Gamma}_{IJ}^0&=\begin{pmatrix} 0&0&0\\0&-\frac{1}{2}(g^\prime_{ij}+mu^{m-1}\psi_{ij})&0\\0&0&0\end{pmatrix},\\
\widetilde{\Gamma}_{IJ}^k&=\begin{pmatrix}0&0&0\\0&\Gamma_{ij}^k+u^m(\Gamma_\psi){}^k_{ij}&\frac{1}{2}g^{kl}(g^\prime_{li}+mu^{m-1}\psi_{li})\\0&\frac{1}{2}g^{kl}(g^\prime_{li}+mu^{m-1}\psi_{li})&0 \end{pmatrix},\\
\widetilde{\Gamma}_{IJ}&=\begin{pmatrix}0&0&0\\0&0&0\\0&0&0 \end{pmatrix}.
\end{aligned}
\end{equation}
From \cref{christoffel symbols for the iterative formula,straight-equations}, we get 
\begin{equation}
\label{iterative-formulas-ijf}
\begin{aligned}
(\widetilde{P}_\phi)_{ij}^{(m)}&\equiv (\widetilde{P}_\phi)_{ij}^{(m-1)}-\frac{1}{2}mu^{m-1}\psi_{ij}^{(m)}\;\bmod\;O(u^m),\\
(\widetilde{F}_{\phi})^{(m)}&\equiv (\widetilde{F}_{\phi})^{(m-1)}
-mu^{m-1}\big(\frac{1}{2}g^{ij}\psi_{ij}^{(m)}-2{\phi}^{(m)}\big)\;\bmod\;O(u^m),\\
(\widetilde{P}_\phi)_{0\infty}^{(m)}&\equiv (\widetilde{P}_\phi)_{0\infty}^{(m-1)}\;\bmod\;O(u^{m-1}),\\
(\widetilde{P}_\phi)_{i\infty}^{(m)}&\equiv (\widetilde{P}_\phi)_{i\infty}^{(m-1)}\;\bmod\;O(u^{m-1}),\\
(\widetilde{P}_\phi)_{\infty\infty}^{(m)}&\equiv (\widetilde{P}_\phi)_{\infty\infty}^{(m-1)}-m(m-1)u^{m-2}\big(\frac{1}{2}g^{kl}\psi^{(m)}_{kl}-{\phi}^{(m)}\big)\;\bmod\;O(u^{m-1}).
\end{aligned}
\end{equation}
From \cref{iterative-formulas-ijf}, we conclude that there is no obstruction to solving $(\widetilde{P}_\phi)_{ij}$ and $\widetilde{F}_{\phi}$ to order $O(u^{\infty})$ at $u=0$.

We now show that $(\widetilde{P}_\phi)_{I\infty}\equiv 0\;\bmod\;O(u^\infty)$ for $I=\{0,i,\infty\}$.
We know from \cref{000ivanishing} that $(\widetilde{P}_\phi)_{00},(\widetilde{P}_\phi)_{0i}\equiv 0\;\bmod\;O(u^\infty)$.
Now let us assume that 
\begin{equation}
\begin{aligned}
\label{ijsvanishing}
(\widetilde{P}_\phi)_{ij}&\equiv 0\;\bmod\;O(u^{m}),\\ \widetilde{F}_{\phi}&\equiv 0\;\bmod\;O(u^{m}),\\
(\widetilde{P}_\phi)_{I\infty}&\equiv 0\;\bmod\;O(u^{m-2}).
\end{aligned}
\end{equation}
Taking $I=0,i,\infty$ in \cref{weighted bianchi}, computing$\mod O(u^{m-1})$ at $u=0$, and simplifying using \cref{000ivanishing,ijsvanishing}, yields
\begin{equation}
\label{weightedbianchiequations}
\begin{aligned}
2(\widetilde{P}_\phi)_{0\infty}&\equiv 0\;\bmod\;O(u^{m-1}),    \\
2(\widetilde{P}_\phi)_{i\infty}&\equiv 0\;\bmod\;O(u^{m-1}) ,    \\
2(\frac{1}{2}g^{kl}g^\prime_{kl}-{\phi}^\prime-\lambda v)(\widetilde{P}_\phi)_{0\infty}\\
+2(\widetilde{P}_\phi)_{\infty\infty}
+2\nabla_{\phi}^k(\widetilde{P}_\phi)_{k\infty}&\equiv 0\;\bmod\;O(u^{m-1}). 
\end{aligned}
\end{equation}
We see from \cref{weightedbianchiequations} that $(\widetilde{P}_\phi)_{I\infty}\equiv 0\;\bmod\;O(u^{m-1})$ for $I=\{0,i,\infty\}$.
\end{proof}
\begin{proof}[Proof of \cref{ambient-metric-theorem}]
Using \cref{pullback-normal,normal-ambient-metric-exists}, we complete the proof of \cref{ambient-metric-theorem}.
\end{proof}
\subsection{The first few terms in the expansion of $(g_u,\phi_u)$} We now calculate the first few terms of the Taylor expansions of $g_u$ and $\phi_u$.

Using \cref{straight-equations}, we write $(\widetilde{P}_\phi)_{IJ},(\widetilde{F}_\phi)\equiv 0\bmod O(u^\infty)$ at $u=0$ as
\begin{subequations}
\begin{align}
\begin{split}
\label{ijcomponent}
(\widetilde{P}_\phi)_{ij}&=(P_{\phi})_{ij}-\frac{1}{2} g^\prime_{ij}\equiv 0\bmod O(u^\infty), \end{split}\\
\begin{split}
\label{scomponent}
\widetilde{F}_\phi&=-\frac{1}{2} g^{ij}g^\prime_{ij}+F_{\phi}+2 \phi^\prime - n \lambda\equiv 0\bmod O(u^\infty),
\end{split}\\
\begin{split}
\label{0inftycomponent}
(\widetilde{P}_\phi)_{0\infty}&\equiv 0\bmod O(u^\infty),   \end{split}\\
\begin{split}
\label{iinftycomponent}
(\widetilde{P}_\phi)_{i\infty}&=\frac{1}{2}[(\delta_{\phi}g^\prime)_i-\nabla_i(g^\prime)_k{}^k
+2\nabla_i\phi^\prime]\equiv 0\bmod O(u^\infty), \end{split}\\
\begin{split}
\label{inftyinftycomponent}
(\widetilde{P}_\phi)_{\infty\infty}&=-\frac{1}{2}g^{kl}g^{\prime\prime}_{kl}+\frac{1}{4}(g^\prime)^{kl}(g^\prime)_{kl} +\phi^{\prime\prime}\equiv 0\bmod O(u^\infty). \end{split}    
\end{align}
\end{subequations}
From \cref{ijcomponent,scomponent}, we have at $u=0$ 
\begin{equation}
\label{g'phi'}
\begin{aligned}
g_{ij}^\prime&=2(P_{\phi})_{ij},\\ \phi^\prime&=-Y_{\phi}.
\end{aligned}
\end{equation}
On differentiating \cref{ijcomponent}, and using the first variation formulas of $\mathrm{Ric}$ and $\nabla^2\phi$~(cf.\ \cite{Besse}, pg. 62), 
we get at $u=0$
\begin{equation}
\label{g''}
\begin{aligned}
g^{\prime\prime}_{ij}=-2(B_{\phi})_{ij}+2\pt_{i}^p\pt_{pj}.
\end{aligned}
\end{equation}
From \cref{inftyinftycomponent}, we get
\begin{equation}
\label{phi''}
{\phi}^{\prime\prime}=-\mathrm{tr}_g B_{\phi}.
\end{equation}
On differentiating \cref{inftyinftycomponent} and again using the relevant variation formulas, we get
\begin{equation*}
\frac{1}{2}g^{ij}g^{\prime\prime\prime}_{ij}-\phi^{\prime\prime\prime}=4(P_\phi)_i^k (P_\phi)_k^j (P_\phi)_j^i. \end{equation*}
Note that if we don't assume $u=0$ in \cref{ijcomponent,scomponent}, we get 
\begin{align*}
g'_{ij}&=2(1-\lambda u)^{-1}P_\phi,\\
\phi'&=-(1-\lambda u)^{-1}Y_\phi.
\end{align*}
This observation will be useful in constructing the Ricci flow vector field in \cref{sec:w-funct}.


\section{The $\mathcal{F}$--gradient Ricci flow}
\label{sec:ricci}
In this section, we prove \cref{ricci-flow,h-embedding,ricci-equivalent-normal,einstein-long-term-behavior}. 

We first prove the existence of the gradient Ricci flow in the ambient space $(\widetilde{\mathcal{G}},\widetilde{g},\widetilde{\phi},\lambda=0)$ along the vector field 
\begin{equation}
\label{def-X}
X:=(v_{1,\phi})_u\partial_v-\partial_u.
\end{equation}
\begin{proof}[Proof of \cref{ricci-flow}]
For $\lambda=0$, we deduce from \cref{normalambientmetric} that the ambient metric is of the form
\begin{align*}
\widetilde{g}&= 2 dudv+g_u, \quad \widetilde{\phi}=\phi_u-v.
\end{align*}
We obtain the following from \cref{ijcomponent,scomponent}:
\begin{equation*}
\begin{aligned}
g^\prime_u&=2 (\mathrm{Ric}_\phi)_u\bmod O(u^\infty),\\ \phi_u^\prime&=\frac{1}{2}(R_u+|\nabla\phi_u|^2)\bmod O(u^\infty).
\end{aligned}
\end{equation*}
Hence, we have
\begin{equation}
\label{ricci-flow-eqns}
\begin{aligned}
(L_X g_u)_{ij}&=-2 [(\mathrm{Ric}_\phi)_u]_{ij}\bmod O(u^\infty),\\
L_X \widetilde{\phi}&=-(R_u+\Delta\phi_u)\bmod O(u^\infty).\qedhere
\end{aligned}
\end{equation}
\end{proof} 
Note that 
\begin{equation}
\label{volume-constant-ricci-flow}
L_X [e^{-\widetilde{\phi}}\mathrm{dvol}_{\widetilde{g}}]=0\bmod O(u^\infty).
\end{equation}

We now prove that, given a gradient Ricci flow spacetime, we can construct a unique global weighted ambient half-space that contains it. This result holds for general $\lambda$. Later, we will consider the $\lambda=0$ case, which will allow us to study the monotonicity of the $\mathcal{F}_{k,\phi}$ functionals under the Ricci flow. 
\begin{proof}[Proof of \cref{h-embedding}]
At any point $m\in\mathcal{M}$, the gradient Ricci flow spacetime reduces to a gradient Ricci flow in the usual sense, because the time function $\mathfrak{t}$ forms part of a chart $(x^i,t)$ near $m$ for which the coordinate vector field $\partial/\partial t$ coincides with $\partial_{\mathfrak{t}}$~(cf.\ \cite{KleinerLott2017}, pg. 69). On $ \mathbb{R}\times \mathcal{M}$, we consider the coordinates $\{v,x^i,u\}$, where $\{x^i,u\}$ are the coordinates on $\mathcal{M}$ and $v$ the coordinate on $\mathbb{R}$. Here $u$ is such that $\partial_t=-(1-\lambda u)\partial_u$. We use $0$ to label the $v$-component, $\infty$ to label the $u$-component, and lower case Latin to label coordinates on $M$. Let $(\widetilde{g},\widetilde{\phi})$ be such that in these coordinates,
\begin{equation*}
\begin{aligned}
\widetilde{g}&=2dudv+g_u-(1-\lambda u)^2(R_\phi)_udu^2,\\
\widetilde{\phi}&=\phi_u-v+\lambda uv,  \end{aligned} \end{equation*}
where $(R_\phi)_u=R_u+2\Delta_u\phi_u-|\nabla_u\phi_u|^2+2\lambda(\phi_u-n)$. 
Hence, we have 
\begin{equation}
\label{inv-g-half-space}    \widetilde{g}^{IJ}=\begin{pmatrix}
        (1-\lambda u)^2(R_\phi)_u &0&1\\ 0& g_u^{ij}&0\\ 1&0&0
    \end{pmatrix}.
\end{equation}
From \cref{inv-g-half-space}, we calculate that 
\begin{equation}
\label{christoffel-half-space}
\begin{aligned}
 \Gamma_{IJ}^0&=\begin{pmatrix}
     0&0&0\\0&-\frac{1}{2}g'_{ij}&-\frac{1}{2}(1-\lambda u)^2\partial_i (R_\phi)_u\\
     0&-\frac{1}{2}(1-\lambda u)^2\partial_i (R_\phi)_u&-\frac{1}{2}[(1-\lambda u)^2 R_\phi]'_u
 \end{pmatrix},\\
 \Gamma_{IJ}^k&=\begin{pmatrix}
     0&0&0\\0&\Gamma_{ij}^k&\frac{1}{2}g^{kl}g'_{il}\\ 0&\frac{1}{2}g^{kl}g'_{il}&\frac{1}{2}(1-\lambda u)^2g^{kl}\partial_l (R_\phi)_u
 \end{pmatrix},\\
 \Gamma_{IJ}^\infty&=\begin{pmatrix}
     0&0&0\\0&0&0\\0&0&0
 \end{pmatrix}.
\end{aligned}    
\end{equation}
From \cref{christoffel-half-space}, we calculate that 
\begin{equation}
\label{half-space-ric-0}(\widetilde{P}_\phi)_{00}=(\widetilde{P}_\phi)_{0i}=(\widetilde{P}_\phi)_{ij}=\widetilde{F}_\phi=(\widetilde{P}_\phi)_{0\infty}=0.
\end{equation}
From \cref{christoffel-half-space,weighted bianchi,half-space-ric-0}, we calculate that \[(\widetilde{P}_\phi)_{i\infty}=(\widetilde{P}_\phi)_{\infty\infty}=0.\qedhere\]
\end{proof}
Note that the ambient space constructed above is ambient-equivalent to a normal global weighted ambient space. Hence, given a gradient Ricci flow spacetime $\mathcal{M}$, we can embed it in a normal global weighted ambient space. 

We prove that all global ambient half-spaces are ambient-equivalent to a normal ambient half-space.
\begin{proof}[Proof of \cref{ricci-equivalent-normal}]
There exists a vector field $V \in T\widetilde{\mathcal{G}}$ such that 
\begin{equation}
\label{Vconditions'}
\begin{aligned}
\widetilde{g}(V,T)&=1,\\
\widetilde{g}(V,X)&=0\text{ for all }X\in\mathcal{M}_z,\\
\widetilde{g}(V,V)&=0.
\end{aligned}
\end{equation} 
In particular, let 
\begin{equation*}
\begin{aligned}
V&=V^0\partial_v+V^i\partial_{x^i}+V\partial_u.
\end{aligned}    
\end{equation*}
\cref{Vconditions'} becomes
\begin{equation}
\label{Vequationsproof'}
\begin{aligned}
-V^0&=1,\\
\widetilde{g}_{ij}V^j&=0\\
2 V^{0} V^{\infty}+ \widetilde{g}_{i j} V^{i} V^{j}-R_\phi\left(V^{\infty}\right)^{2}&=0.
\end{aligned}    
\end{equation}
We can use \cref{Vequationsproof'} to solve uniquely for $V^0,V^i$ and $V$. The dilation invariance and smoothness follow from the properties defining a pre-ambient space. 

Let $\psi_t$ be the one-parameter family of diffeomorphisms generated by $V$. For $z \in \mathcal{G}$, consider the map $(z,t)\to \psi_t (z)$. Let the metric at $(z,t)$ be $\psi_t^*(z)(\widetilde{g})$. 
Note that $\psi_t(z)$ is defined for all $t\in I$, where $I$ is the co-domain of the time function $t$ on the normal spacetime. Since $\widetilde{g}$ and $V_z$ are homogeneous with respect to the dilations, we may take $\mathcal{G}\times I$ to be dilation-invariant. Thus, $\psi: \mathcal{G}\times I \rightarrow \widetilde{\mathcal{G}}$ is a smooth map that commutes with dilations.

Since $(\widetilde{\mathcal{G}}, \widetilde{g})$ is a pre-ambient space for $(M,[g])$, so is $\left(\mathcal{U}, \psi^* \widetilde{g}\right)$. Also, we have:
$$
\begin{aligned}
& \left(\psi^* \widetilde{g}\right)\left(\partial_t, T\right)=1, \\
& \left(\psi^* \widetilde{g}\right)\left(\partial_t, X\right)=0 \text { for } X \in T M, \\
& \left(\psi^* \widetilde{g}\right)\left(\partial_t, \partial_t\right)=0 .
\end{aligned}
$$
Together with \cref{iv}, these equations show that $\psi^* \widetilde{g}=\mathbf{g}_0+2 d v dt$ when $t=0$. This establishes the existence the normal global weighted ambient space. The uniqueness follows from the fact that the above construction of $\psi$ is forced. If $\psi$ is any diffeomorphism with the required properties, then at every point in $\widetilde{\mathcal{G}}, \psi_*\left(\partial_t\right)$ is a $g$-transversal for $\widetilde{g}$, so must be $V$. These requirements uniquely determine $\psi$ on $\mathcal{G}\times I$.
\end{proof}
\begin{proof}[Proof of \cref{every-global-ambient}]
This follows from \cref{h-embedding,ricci-equivalent-normal}.    
\end{proof}
We define \[p_{k,\phi}=k!v_{k,\phi}.\]
We recall from \cref{normalambientmetric} that for $\lambda=0$, the unique normal ambient metric $(\widetilde{g},\widetilde{\phi})$ is of the form 
\begin{equation}
\label{normalambientlambda=0}
\widetilde{\phi}=\phi_u-v,\quad \widetilde{g}=2dudv+g_u.
\end{equation}
At a point $p\in\widetilde{\mathcal{G}}$, we define 
\begin{equation}
\label{newpkdefinition}
p_{k,\phi}(v,x,u):=\partial_u^k\frac{e^{-\widetilde{\phi}}(\mathrm{det}\, \widetilde{g})^{1/2}}{\restr{e^{-\widetilde{\phi}}(\mathrm{det}\, \widetilde{g})^{1/2}}{p}},
\end{equation} where $\partial_u^k$ refers to $\underbrace{\partial_u\dotsm\partial_u}_{k\text{ times}}$. We have from \cref{normalambientlambda=0,newpkdefinition,volcoefficientsdef} that
\begin{equation*}
 \restr{p_{k,\phi}(v,x,u)}{p}=p_{k,\phi}.   
\end{equation*}
\begin{lemma}
For all $k\in\mathbb{N}$, we have
\begin{equation}
\label{deriv-equation}
    \restr{[L_X p_{k,\phi}(v,x,u)]}{p}=-p_{k+1,\phi}+p_{1,\phi} p_{k,\phi}.
\end{equation}    
\end{lemma}
\begin{proof}
We deduce from \cref{newpkdefinition,normalambientlambda=0} that 
\begin{equation}
\label{pkderivativeexpression}
L_X p_{k,\phi}(v,x,u)= -p_{k+1,\phi}(v,x,u)+v_{1,\phi} p_{k,\phi}(v,x,u).  
\end{equation}
On restricting \cref{pkderivativeexpression} to $p$ and noting that $v_{1,\phi}=p_{1,\phi}$, we get
\[\restr{L_X p_{k,\phi}(v,x,u)}{p}=-p_{k+1,\phi}+p_{1,\phi}p_{k,\phi}.\qedhere\]
\end{proof}
Using \cref{deriv-equation}, we recover Perelman's result~\cites{Perelman1} that $\mathcal{F}_{1,\phi}$ is monotonic under the gradient Ricci flow.
\begin{lemma}
\label{f-increasing}
$\mathcal{F}_{1,\phi}$ is monotone increasing under the gradient Ricci flow.
\end{lemma}
\begin{proof}
From \cref{v-formulas,deriv-equation,volume-constant-ricci-flow}, we get 
\[L_X\mathcal{F}_{1,\phi}=\int|P_\phi|^2\mathrm{dvol}_{\widetilde{g}}\geq 0.\qedhere\]
\end{proof}
The \emph{positive weighted elliptic $k$-cone} $\Gamma_k^{\infty,+}$ is the set~\cite{Case2014sd}
\[
\Gamma_k^{\infty,+}=\left\{(g,\phi)\in Met(M) \times C^\infty(M):p_{j,\phi}>0 \text { for all } j \in\{1, \ldots, k\}\right\}
\]
We see below that $\mathcal{F}_{k,\phi}$ is monotone increasing under the gradient Ricci flow if $(g,\phi)$ lies in $\Gamma_k^{\infty,+}$, and satisfies the natural Newton-type inequality $p_{k-1,\phi}p_{k+1,\phi}\leq (p_{k,\phi})^2$.
\begin{lemma}
\label{positive-cone-increasing}
Let $(g,\phi)\in \Gamma_{s}^{\infty,+}$. Then $\mathcal{F}_{k,\phi}$ is monotone increasing under the gradient Ricci flow for $k\in\{1,\dotsm,s+1\}$, as long as $\mathcal{F}_{k-1,\phi}$ is monotone increasing, and
\begin{equation}
\begin{aligned}
\label{pk-1pk+1}
p_{k-1,\phi}p_{k+1,\phi}\leq (p_{k,\phi})^2.
\end{aligned}
\end{equation}
\end{lemma}
Note that, unlike for $\sigma_{k,\phi}$~\cites{Case2016v,Case2014s} Newton-type inequalities of the form \cref{pk-1pk+1} are not always satisfied by $p_{k,\phi}$. 
\begin{example}
We deduce from \cref{einstein-volume-coefficients} that $(S^3,c)$ does not satisfy the inequality $p_{1,\phi}p_{3,\phi}\leq (p_{2,\phi})^2$.
\end{example}
However, we know from \cref{v-formulas} that this is always true for $k=1$, as $-p_{2,\phi}+p^2_{1,\phi}=|P_\phi|^2$.
\begin{proof}[Proof of \cref{positive-cone-increasing}] We assume that $\mathcal{F}_{k-1,\phi}$ is monotone increasing, and hence we deduce from \cref{deriv-equation} that $-p_{k,\phi}+p_{1,\phi}p_{k-1,\phi}\geq 0$. We deduce using \cref{pk-1pk+1} that
\[p_{k-1,\phi}p_{k+1,\phi}\leq (p_{k,\phi})^2\leq p_{1,\phi}p_{k-1,\phi}p_{k,\phi}.\]
As $p_{k-1,\phi}>0$, we have 
\[-p_{k+1,\phi}+p_{1,\phi}p_{k,\phi}\geq 0.\qedhere\]
\end{proof}
\begin{proof}[Proof of \cref{einstein-long-term-behavior}] 
Recall from \cref{einstein-ambient-metric-2} that $g_u=(1+4\mu u)g_0$ and $Ric_u=\mu g_0$, where $g_u$ and $Ric_u$ are calculated from the ambient construction for an Einstein manifold. Hence, we have \[\mathrm{Ric}_u=\mu(1+4\mu u)^{-1}g_u.\] Also recall from \cref{einstein-ambient-metric-2} that $(g_u,\phi_u)$ is an Einstein metric measure structure for all $u>-\frac{1}{4\mu}$. Hence, it follows from \cref{einstein-volume-coefficients} that \[v_{k,\phi}(u)=\frac{1}{k!}\mu^k(1+4\mu u)^{-k}\prod\limits_{j=0}^{k-1}(n-4j).\] Now recall that we have a formal solution to the gradient Ricci flow along the vector field $-\partial_u+v_{1,\phi}\partial_v$. Hence, $u$ tends to $-\frac{1}{4\mu}$ along the gradient Ricci flow, and \[
\lim\limits_{u\to -\frac{1}{4\mu}}\mu(1+4\mu u)^{-1}\to\infty.\]Note that $\prod\limits_{j=0}^{k-1}(n-4j)$ remains unchanged. Therefore $v_{k,\phi}\to \pm \infty$ depending upon the sign of $\prod\limits_{j=0}^{k-1}(n-4j)$.
\end{proof}
\begin{definition}
We define
\begin{equation}
\label{modified-elliptic-cone}
    \Lambda^{\infty,-}=\{(g,\phi):(-1)^{k+1}p_{k,\phi}\geq 0\text{ for all }k\in\mathbb{N}\}.
\end{equation}
Note that $\Lambda^{\infty,-}$ is analogous to the negative elliptic $\infty$-cone, except that the signs for each $p_{k,\phi}$ have been reversed. Moreover, we don't require strict inequality. 
\end{definition}An example of $(g,\phi)\in\Lambda^{\infty,-}$ is $(S^3,c)$, as can be seen from \cref{einstein-volume-coefficients}.\\

If $(g,\phi)\in \Lambda^{\infty,-}$, clearly $(-1)^{k+1}[-p_{k+1,\phi}+p_{1,\phi}p_{k,\phi}]\geq 0$. Hence, as we show below, the gradient Ricci flow preserves $\Lambda_k^{\infty,-}$. Moreover, if $(g,\phi)\in \Lambda^{\infty,-}$, then $(-1)^{k+1}\mathcal{F}_{k,\phi}$ is increasing along the gradient Ricci flow for all time.

\begin{lemma}
\label{lemma-shifted-cone}
If $(g,\phi)\in \Lambda^{\infty,-}$, then the gradient Ricci flow preserves $\Lambda^{\infty,-}$. Moreover, for all $k\in \mathbb{N}$, $(-1)^{k+1}\mathcal{F}_{k,\phi}$ is monotone increasing with the gradient Ricci flow for all time. 
\end{lemma}
\begin{proof}
We see using \cref{deriv-equation,modified-elliptic-cone} that 
\begin{equation*}
(-1)^{k+1}L^s_X p_{k,\phi}\geq 0
\end{equation*}
for all $s\in\mathbb{N}\cup \{0\}$.
\end{proof}

\section{The $\mathcal{W}$--gradient Ricci flow} 
\label{sec:w-funct}
In this section, we prove \cref{w-theorem,einstein-w-flow}.

Consider the following system of equations studied by Perelman~\cite{Perelman1}*{(3.3)}
\begin{equation}
\label{eq-1}
\begin{aligned}
\partial_t g_{ij}&=-2\mathrm{Ric}_{ij},\\
\partial_t \phi&=-R-\Delta\phi+|\nabla\phi|^2+\frac{n}{2\tau},\\
\partial_t\tau&=-1.
\end{aligned}    
\end{equation}
Pulling $(g,\phi,\tau)$ in \cref{eq-1} back by the family of diffeomorphisms generated by $\nabla\phi$, we get the $\mathcal{W}$--gradient Ricci flow, defined as
\begin{equation}
\label{eq-2}
\begin{aligned}
 \partial_t g_{ij}&=-2(\mathrm{Ric}+\nabla^2\phi),\\
 \partial_t \phi&=-R-\Delta\phi+\frac{ n}{2\tau},\\
 \partial_t \tau&=-1.
\end{aligned}
\end{equation}
Perelman proves that the $\mathcal{W}$ functional is monotone increasing under the Ricci flow, as defined by \cref{eq-1}. He does so by proving that $\mathcal{W}$ is monotone increasing under the $\mathcal{W}$--gradient Ricci flow defined by \cref{eq-2}, and that it is unchanged when we pull back $(g,\phi,\tau)$ by the one-parameter family of diffeomorphisms generated by $\nabla\phi$. 

Let $\tau>0$.
We now show that there exists a vector field in the space $(\widetilde{\mathcal{G}},\widetilde{g}(s),\widetilde{\phi}(s),(2\tau(s))^{-1})\times (-\epsilon,\tau)$, along which we have a solution $\mathcal{W}$--gradient Ricci flow. 
\begin{proof}[Proof of \cref{w-theorem}]
In an ambient space $(\widetilde{\mathcal{G}}, \widetilde{g},\widetilde{\phi},(2\tau)^{-1})$, consider the vector field
\begin{equation*}
\label{lambda-vector-field}
X:=(1-(2\tau)^{-1} u)^{-1}(v_{1,\phi})_u\partial_v-(1-(2\tau)^{-1} u)\partial_u,
\end{equation*}
where $(v_{1,\phi})_u=-\frac{1}{2}[R+2\Delta\phi-|\nabla\phi|^2+\tau^{-1} (\phi-n)]$.
We have 
\begin{equation}
\label{eq-3}
\begin{aligned}
 L_X g_{ij}&=-2(\mathrm{Ric}+\nabla^2\phi)+\frac{1}{\tau}g,\\
L_X \phi&=-R-\Delta\phi+\frac{ n}{2\tau}.
\end{aligned}
\end{equation}
Now consider $(\widetilde{\mathcal{G}},\widetilde{g}(s),\widetilde{\phi}(s),{(2\tau(s))}^{-1}) \times (-\epsilon,\tau)$, with $s\in(-\epsilon,\tau)$ being the parameter. For each $s$, we have that $(\widetilde{\mathcal{G}},\widetilde{g}(s),\widetilde{\phi}(s),({2\tau(s)})^{-1})\times \{s\}$ is a normal global weighted ambient space. Given a manifold with density \[(M^n,g,\phi,\lambda)\subset(\widetilde{\mathcal{G}},\widetilde{g}(s),\widetilde{\phi}(s),(2\tau(s))^{-1}),\] we have local normal coordinates $\{v,x^i,u(s)\}$, and a vector $X$, such that \cref{eq-3} is satisfied. In terms of the set of local coordinates $\{v,x^i,u(s),s\}$ on $(\widetilde{\mathcal{G}},\widetilde{g}(s),\widetilde{\phi}(s),{(2\tau(s))}^{-1})\times (-\epsilon,\tau)$, let
\begin{equation}
\label{eq-4}
\begin{aligned}
 \tau(s)&=(1-s/\tau)\tau,\\
 g_{u(s)}(s)&=(1-s/\tau)g_u,\\
 \phi_{u(s)}(s)&=\phi_u.
\end{aligned}
\end{equation}
Note that $(g_{u(s)}(s),\phi_{u(s)}(s))$ does indeed help form the normal ambient metric for each $(\widetilde{\mathcal{G}},\widetilde{g}(s),\widetilde{\phi}(s),(2\tau(s))^{-1})$. In particular, $u(s)$ scales as $(1-s/\tau)u$, while $v$ and $x^i$ remain constant with $s$. We have from \cref{eq-4},
\begin{equation}
\label{tau0-derivative-eqns}
\begin{aligned}
{\partial_s}\tau(s)&=-1,\\
{\partial_s} g_{u(s)}(s)&=-g_u=-\frac{1}{\tau(s)}g_{u(s)}(s).
\end{aligned}    
\end{equation}
From \cref{tau0-derivative-eqns}, for 
\begin{equation*}
Y:=X+\partial_s,
\end{equation*}
we have for all $s$,
\begin{equation*}
\begin{aligned}
 L_Y g_{ij}&=-2(\mathrm{Ric}+\nabla^2\phi),\\
 L_Y \phi&=-R-\Delta\phi+\frac{n}{2\tau},\\
 L_Y \tau&=-1,
\end{aligned}
\end{equation*}
which is the same as \cref{eq-1}. Therefore, the gradient Ricci flow is generated by $Y$ in $\widetilde{\mathcal{G}}\times (-\epsilon,\tau)$. 
\end{proof}
Note that the coordinate $u\leq 0$ in these ambient half-spaces. Hence, $(1-(2\tau(s))^{-1}u)\geq 1$, and the vector $Y$ makes sense throughout the space. 
Also, observe that the gradient Ricci flow spacetime, as defined in \cref{def-spacetime}, lies along $\partial_t$ and hence transverse to each ambient space $(\widetilde{\mathcal{G}},\widetilde{g}(s),\widetilde{\phi}(s),(2\tau(s))^{-1})$.

We have \[\mathcal{W}_{k,\phi}=\int \tau^k v_{k,\phi}\subba.\]Perelman proved~\cite{Perelman1} $\mathcal{W}(cg,\phi,c\tau)=\mathcal{W}(g,\phi,\tau)$. We prove the same property for all $\mathcal{W}_{k,\phi}$'s.
\begin{lemma}
We have
\begin{equation}
\label{w-invariant}
    \mathcal{W}_{k,\phi}(cg,\phi,c\tau)=\mathcal{W}_{k,\phi}(g,\phi,\tau)
\end{equation}
for all $k\geq 1$.
\end{lemma}
\begin{proof}
Consider the re-scaling $(\overline{g},\overline{\phi},\overline{\tau})=(c g,\phi,c\tau)$. The normal coordinates $(\overline{v},\overline{x}^i,\overline{u})$ and $(v,x^i,u)$ are related by the diffeomorphism $\overline{v}=v,\overline{x}^i=x^i$ and $\overline{u}=cu$. As $v_\phi$ is also scale invariant, we get $\overline{v}_{k,\phi}=c^{-k}v_{k,\phi}$. Hence $\tau^k v_{k,\phi}$ is invariant under the re-scaling. It is easy to check that $(4\pi\tau)^{-\frac{n}{2}}e^{-\phi}\mathrm{dvol}_g$ is invariant under this scaling too, thereby implying the same for $\mathcal{W}_{k,\phi}(g,\phi,\tau)$.
\end{proof}
We infer from \cref{w-invariant} that
\begin{equation*}
    \mathcal{W}_{k,\phi}((1-s/\tau)g,\phi,(1-s/\tau)\tau)=\mathcal{W}_{k,\phi}(g,\phi,\tau).
\end{equation*}
Hence, the $\mathcal{W}$--functional remains unchanged along $\partial_s$. Therefore, to study the evolution of the $\mathcal{W}$--functional along the Ricci flow, we only need to study its evolution along $X$. It follows that we have 
\begin{equation*}
\label{W-evolution-ricci-flow}
    \partial_t \mathcal{W}_{k,\phi}=\frac{1}{k!}\int_M \tau^k (-p_{k+1,\phi}+p_{1,\phi}p_{k,\phi})(4\pi\tau)^{-n/2}e^{-\phi}\mathrm{dvol}_g,
\end{equation*}
where $p_{j,\phi}=j! v_{j,\phi}$.
\begin{lemma}
\label{w-increasing}
$\mathcal{W}_{1,\phi}$ is monotone increasing along the gradient Ricci flow.
\end{lemma}
\begin{proof}
    We see from \cref{v-formulas,deriv-equation} that 
\begin{equation}
\label{ric-derivative-w1}
L_X \mathcal{W}_{1,\phi}=\int_M |P_\phi|^2\subba\geq 0.\qedhere
\end{equation}
\end{proof}
We note from \cref{ric-derivative-w1} that $\mathcal{W}_{1,\phi}$ is not strictly monotone increasing if and only if $P_\phi=0$, which is equivalent to the fact that $(g,\phi)$ is a shrinking gradient Ricci soliton. 

We now prove that all $\mathcal{W}_{k,\phi}$ functionals are constant along the Ricci flow for all $k$ for shrinking solitons.
\begin{proof}[Proof of \cref{einstein-w-flow}]
For a shrinking soliton, we calculate that
\begin{align*}
g(t)&=\left(1-t/\tau\right)g,\\
\phi(t)&=\phi,\\
\tau(t)&=\tau-t,
\end{align*}
is the solution to \cref{eq-2} for shrinking solitons that satisfy \[\mathrm{Ric}+\nabla^2\phi=\frac{1}{2\tau}g.\]
Also,
\begin{equation*}
\label{vphi-einstein-w}
v_\phi=e^{au},    
\end{equation*}
where $a=Y_\phi(t)$. It follows from \cref{g'phi'} that for shrinking solitons, $Y_\phi(t)=(\tau-t)^{-1}c'$ for some $c'\in\mathbb{R}$. Hence, \[[\tau(t)]^k v_{k,\phi}(t)=\frac{1}{k!}(c')^k,\] which implies that $\mathcal{W}_{k,\phi}$ remains constant for all $k$ for shrinking solitons. 

The converse follows from \cref{ric-derivative-w1}.
\end{proof}
Note that $c'\leq 0$~\cites{Yokota2008add}. Hence, $\mathcal{W}_{k,\phi}$ are of the sign $(-1)^k$ for shrinking solitons.

\cref{positive-cone-increasing,lemma-shifted-cone} are true for $\mathcal{W}_{k,\phi}$ too. We re-state them here without proof for completeness.
\begin{lemma}
Let $(g,\phi)\in \Gamma_{s}^{\infty,+}$. Then $\mathcal{F}_{k,\phi}$ is monotone increasing under the gradient Ricci flow for $k\in\{1,\dotsm,s+1\}$, as long as $\mathcal{F}_{k-1,\phi}$ is monotone increasing, and
\begin{equation*}
\begin{aligned}
p_{k-1,\phi}p_{k+1,\phi}\leq (p_{k,\phi})^2.
\end{aligned}
\end{equation*}
\end{lemma}
\begin{lemma}
If $(g,\phi)\in \Lambda^{\infty,-}$, then for all $k\in \mathbb{N}$, $(-1)^{k+1}\mathcal{W}_{k,\phi}$ is monotone increasing with the gradient Ricci flow for all time.     
\end{lemma}

\section{Weighted GJMS operators}
\label{sec:gjms}
In this section, we construct weighted GJMS operators, and prove \cref{gjmstheorem}.

The space of conformal densities of $\cmG$ of weight $w$, denoted as $\mathcal{E}[w]$, is defined as \[\mathcal{E}[w]=\{f:\mathcal{C}\to C^\infty(M)\mid f(\phi+v)(x)=e^{wv}f(\phi)(x)\}.\]
Let $\lambda=0$. Given $f\in C^\infty (M)$, let $e^{wv}f_u\in\mathcal{E}[w]$, such that $f_0=f$.
We have at $u=0$,
\begin{equation}
\label{ambientlaplacian}
\widetilde{\Delta}_{\phi} e^{w v}f_u=e^{w v}\big[\Delta_\phi f_u+\frac{1}{2}R_{\phi} wf_u+(2w+1)f^\prime_u\big]
.    
\end{equation}
For $w=-\frac{1}{2}$, we can rewrite \cref{ambientlaplacian} at $u=0$ as
\begin{equation*}
\widetilde{\Delta}_{\phi} e^{-\frac{1}{2} v}f=e^{-\frac{1}{2} v}\big[\Delta_\phi-\frac{1}{4}R_{\phi} \big]f.
\end{equation*}

We observe that $\widetilde{\Delta}^{k}_{\phi}$ maps $\mathcal{E}[-1/2]\to \mathcal{E}[-1/2]$ for all $k\in\mathbb{N}$. Also, we see from \cref{ambientlaplacian} that $e^{\frac{1}{2}v}\restr{\widetilde{\Delta}^{k}_{\phi}}{u=0} e^{-\frac{1}{2}v}f_u$ is independent of the choice of extension $e^{-\frac{1}{2}v}f_u$ of $e^{-\frac{1}{2}v}f\in \mathcal{E}[-1/2]$ to $\cmG$. Hence, 
\[L_{2k,\phi}f:=e^{\frac{1}{2}v}\restr{\widetilde{\Delta}^{k}_{\phi}}{u=0} e^{-\frac{1}{2}v} f_u\] is a well-defined operator on $C^\infty(M)$.

\cref{ambientlaplacian} implies that \[\widetilde{\Delta}_\phi^ke^{-\frac{1}{2}v}f_u=e^{-\frac{1}{2}v}\big(\Delta_\phi-\frac{1}{4}R_\phi\big)^kf_u.\] Hence, $L_{2k,\phi}=\left(\Delta_\phi-\frac{1}{4}R_\phi\right)^k$, which is formally self-adjoint with the leading order term $\Delta_\phi^k$.

\section{Examples}
\label{sec:examples}
In this section, we provide examples of the ambient space associated with Einstein manifolds, gradient Ricci solitons or locally conformally flat manifolds, thereby proving \cref{examples-theorem,second-examples-theorem}.

First, we study gradient Ricci solitons of the form of \cref{grs-definition}.
\subsection{Gradient Ricci solitons}
\label{gradient-shrinking-solitons}
Gradient Ricci solitons have particularly simple expressions for $g'$ and $\phi'$.
\begin{lemma}
Gradient Ricci solitons satisfy
\begin{equation}
\label{gradient-soliton-g'phi'}
g^\prime_{ij}=0,\quad \phi^\prime= \frac{1}{2}(n\lambda - c) 
\end{equation}
for some $c\in\mathbb{R}$.
\end{lemma}
\begin{proof}
This follows from \cref{grs-definition,g'phi'}.
\end{proof}
Note that in $\mathcal{C}_1$, we have $\phi^\prime=\frac{1}{2}(n\lambda - c) \geq 0$. This is because $v_{k,\phi}\leq 0$ for gradient shrinking solitons in $\mathcal{C}_1$ for all $k\in\mathbb{N}$,~\cites{Yokota2008add,Yokota2008}, and $v_{1,\phi}=-\phi^\prime$. 
\begin{proof}[Proof of \cref{Theorem 1.2-1} of \cref{examples-theorem}]
Let $(g_u,\phi_u)$ be of the form of \cref{quadratic-expression} such that $(g_0,\phi_0)=(g,\phi)$. We obtain from \cref{grs-definition} that $(g_u,\phi_u)=(g,\phi+2^{-1}(n\lambda-c) u)$. Now set $(\widetilde{g},\widetilde{\phi})=(2dudv+g_u,\phi_u-v+\lambda uv)$. Plugging these into \cref{iterative-formulas-ijf} yields $\widetilde{P}_\phi=0$ and $\widetilde{F}_\phi=0$.
\end{proof}
Second, we study weighted locally conformally flat manifolds.
\subsection{Weighted locally conformally flat}
We see from~\cref{lcf-definition} that weighted locally conformally flat manifolds satisfy the relation $\pt_{ij;k}=\pt_{(ij;k)}.$

Note that we can write $g_u$ in \cref{quadratic-expression} as
$$
(g_u)_{i j}=g_{i l} (U_\phi)_{k}^{l} (U_\phi)_{j}^{k}=(U_\phi)_{i k} (U_\phi)_{j}^{k},
$$ where

$$
(U_\phi)_{j}^{i}:=\delta_{j}^{i}+ u(P_\phi)_{j}^{i}.
$$
Set $V_\phi:=(U_\phi)^{-1}$. Then $(V_\phi)_{k}^{i} (U_\phi)_{j}^{k}=\delta_{j}^{i}$. Note that $(U_\phi)_{i j}$ and $(V_\phi)_{i j}$ are both symmetric. Additionally,
\begin{equation}
\begin{aligned}
\label{u-vrelations}
(V_\phi)_{i}^{k} (g_u)_{k j}&=(U_\phi)_{i j},\\
(g_u)_{i j}^{\prime}&=2 (P_\phi)_{i k} (U_\phi)_{j}^{k},\\
u(V_\phi)_i^k(P_\phi)_k^j&=\delta_i^j-(V_\phi)_i^j.
\end{aligned}
\end{equation}

We now relate the Levi-Civita connections ${ }^{g_u} \nabla$ and ${ }^{g} \nabla$, and the curvature tensors ${ }^{g_u} R_{i j k l}$ and ${ }^{g} R_{i j k l}$.
\begin{lemma}
Let $(g_u,\phi_u)$ and $(g,\phi)$ be defined as in \cref{examples-theorem}. The Levi-Civita connections of $g_u$ and $g$ are related by
\begin{equation}
\label{connection-conformally-flat}
{ }^{g_u} \nabla_{i} \eta_{j}={ }^{g} \nabla_{i} \eta_{j}-u (V_\phi)_{l}^{k} (P_\phi)_{i; j}^{l} \eta_{k},
\end{equation} and the curvature tensors by
\begin{equation*}
\label{curvature-conformally-flat}
{ }^{g_u} R_{i j k l}={ }^{g} R_{a b k l} (U_\phi)_{i}^{a} (U_\phi)_{j}^{b}.
\end{equation*}
\end{lemma}

\begin{proof}
First note that, since $dP_\phi=0$ for a weighted locally conformally flat manifold, the right hand side of \cref{connection-conformally-flat} defines a torsion-free connection.

We now show that $g_u$ is parallel with respect to the connection determined by the right side of \cref{connection-conformally-flat}. On differentiating the metric, we get 
\begin{equation*}
\label{metric-differential}
^{g}\nabla_{k} (g_u)_{i j}=2 u (P_\phi)_{i j; k}+2 u^{2} (P_\phi)_{(i}^{l} (P_\phi)_{j) l; k}.
\end{equation*}
Using the definitions of $U_\phi$ and  $V_\phi$ and the symmetry of $(P_\phi)_{ij;k}$, we get
\begin{align*}
(V_\phi)_{l}^{\infty} (P_\phi)_{k;(i}^{l} (g_u)_{j) m}=(P_\phi)_{k;(i}^{l} (U_\phi)_{j) l}=(P_\phi)_{i j; k}+u (P_\phi)_{k;(i}^{l} (P_\phi)_{j) l}.
\end{align*}
Therefore
$${}^g\nabla_{k} (g_u)_{i j}-2u (V_\phi)_{l}^{\infty} (P_\phi)_{k;(i}^{l} (g_u)_{j) m}=0.$$ 

Combining the previous two paragraphs verifies \cref{connection-conformally-flat}. 

Now set
$$(D_\phi)_{j k}^{i}:=u (V_\phi)_{l}^{i} (P_\phi)_{j; k}^{l},$$ so that $(D_\phi)_{j k}^{i}$ is the difference of ${}^{g}\nabla$ and ${}^{g_u}\nabla$. The difference of the curvature tensors of the connections is given in terms of $D_\phi$ by
$$
{ }^{g_u} R_{m j k l} (g_u)^{i m}-{ }^{g} R_{j k l}^{i}=2 (D_\phi)_{j[l; k]}^{i}+2 (D_\phi)_{j[l}^{c} (D_\phi)_{k] c}^{i};
$$ see \cite[(7.6)]{FeffermanGraham2012}.
We compute that
$$
\begin{aligned}
(D_\phi)_{j l; k}^{i} &=u\left((V_\phi)_{a; k}^{i} (P_\phi)_{j; l}^{a}+(V_\phi)_{a}^{i} (P_\phi)_{j; l k}^{a}\right) \\
&=-u (V_\phi)_{b}^{i} (U_\phi)_{c; k}^{b} (V_\phi)_{a}^{c} (P_\phi)_{j; l}^{a}+u (V_\phi)_{a}^{i} (P_\phi)_{j; l k}^{a} \\
&=-u^{2} (V_\phi)_{b}^{i} (P_\phi)_{c; k}^{b} (V_\phi)_{a}^{c} (P_\phi)_{j; l}^{a}+u (V_\phi)_{a}^{i} (P_\phi)_{j; l k}^{a} \\
&=-(D_\phi)_{c k}^{i} (D_\phi)_{j l}^{c}+u (V_\phi)_{a}^{i} (P_\phi)_{j; l k}^{a}.
\end{aligned}
$$
Therefore, 
\begin{align*}
{ }^{g_u} R_{m j k l} (g_u)^{i m} &={ }^{g} R_{j k l}^{i}+2 u (V_\phi)^{i a} (P_\phi)_{a j;[l k]} \\
&={ }^{g} R_{j k l}^{i}+u (V_\phi)^{i a}\left({ }^{g} R_{a l k}^{b} (P_\phi)_{b j}+{ }^{g} R_{j l k}^{b} (P_\phi)_{a b}\right).
\end{align*}
Since $(V_\phi)_{i}^{k} (g_u)_{k j}=(U_\phi)_{i j}$, we conclude that
\begin{align*}
{ }^{g_u} R_{i j k l} &={ }^{g} R_{j k l}^{b} (g_u)_{b i}+u (U_\phi)_{i}^{a}\left({ }^{g} R_{a l k}^{b} (P_\phi)_{b j}+{ }^{g} R_{j l k}^{b} (P_\phi)_{a b}\right) \\
&={ }^{g} R_{j k l}^{b} (U_\phi)_{b a} (U_\phi)_{i}^{a}+u (U_\phi)^{a}_{i}\left({ }^{g} R_{a l k}^{b} (P_\phi)_{b j}+{ }^{g} R_{j l k}^{b} (P_\phi)_{a b}\right) \\
&=\left[{ }^{g} R_{a b k l}\left(\delta_{j}^{b}+u (P_\phi)_{j}^{b}\right)\right] (U_\phi)_{i}^{a} \\
&={ }^{g} R_{a b k l} (U_\phi)_{j}^{b} (U_\phi)^{a}{ }_{i}.\qedhere
\end{align*}
\end{proof}

\begin{lemma}
The metric \ref{quadratic-expression} is flat for weighted locally conformally flat manifolds with density. 
\end{lemma}
\begin{proof}
Using \cref{christoffel symbols for the undetermined metric}, we calculate the curvature tensor for $(\widetilde{g},\widetilde{\phi})=(2dudv+g_u,\phi-v+\lambda uv)$.
\begin{subequations}
\label{curvatureequations}
\begin{align}
\begin{split}
\nonumber
\widetilde{R}_{IJK0}&=0,
\end{split}\\
\begin{split}
\label{ijkl}
\widetilde{R}_{ijkl}&=R_{ijkl},
\end{split}\\
\begin{split}
\label{ijkinfty}
\widetilde{R}_{ijk\infty}&=\frac{1}{2}[\nabla_i g^\prime_{jk}-\nabla_j g^\prime_{ik}],
\end{split}\\
\begin{split}
\label{rinftyinfty}
\widetilde{R}_{\infty ij\infty}&=\frac{1}{2}[g^{\prime\prime}_{ij}-\frac{1}{2}g^{kl}g^\prime_{lj}g^\prime_{ki}].
\end{split}
\end{align}
\end{subequations}
It is clear from \cref{ijkl} that $\widetilde{R}_{ijkl}=0$.

As $(M^n,g)$ is flat, it is clear from \cref{quadratic-expression,ijkinfty} that $\widetilde{R}_{ijk\infty}=0$.

Differentiating \cref{quadratic-expression} gives us $g^{\prime\prime}_{ij}=2\pt_{ik}\pt^k{}_j$. Also, $g^{ij}=\vt^{il}\vt_l{}^j$. Using \cref{u-vrelations,rinftyinfty}, we see that $\widetilde{R}_{\infty i j \infty}=0$.

Hence, we see that \cref{quadratic-expression} is flat. 
\end{proof}

\begin{proof}[Proof of \cref{Theorem 1.2-2} of \cref{examples-theorem}]

Direct computation using \cref{christoffel symbols for the undetermined metric} immediately yields 
$$
(\widetilde{\nabla}^2 \widetilde{\phi}-\lambda \widetilde{g})_{0 I}=(\widetilde{\nabla}^2 \widetilde{\phi}-\lambda \widetilde{g})_{\infty\infty}=0.
$$

We next prove that $(\widetilde{\nabla}^2\widetilde{\phi})_{i\infty}=0$. Observe that by \cref{basic-definitions-2}, \begin{equation}\label{phiyphim}-\nabla_i (Y_{\phi})=(P_\phi)_i^k \nabla_k \phi,\end{equation} and by \cref{u-vrelations}, \[\frac{1}{2}g^{kl}g^\prime_{li}=(V_\phi)^k_s (P_\phi)^s_i.\] We then compute that 
\begin{align*}
(\widetilde{\nabla}^2 \widetilde{\phi})_{i\infty}&=-\partial_i \left(Y_{\phi}\right)-\widetilde{\Gamma}_{i\infty}^k \partial_k\left(\phi-u Y_{\phi}\right)\\
&=(P_\phi)^k_i \partial_k \phi-(V_\phi)^k_s(P_\phi)^s_i\partial_k \phi-u(V_\phi)^k_s (P_\phi)^s_i (P_\phi)^l_k \partial_l \phi.
\end{align*}
The conclusion follows from the identity $-u (P_\phi)^l_k=\delta^l_k-(U_\phi)^l_k.$

We now compute $(\widetilde{\nabla}^2\widetilde{\phi}-\lambda \widetilde{g})_{ij}$. 

Note that 
\begin{align*}
(\widetilde{\nabla}^2\widetilde{\phi})_{ij}&={}^{g_u}\nabla^2_{ij}\phi_u+\frac{1}{2}(\lambda u-1)g^\prime_{ij}\\
&={}^{g_u}\nabla^2_{ij}(\phi-uY_{\phi})+(\lambda u-1)(P_\phi)_{ik}(U_\phi)^k_j.
\end{align*}

Also, from \cref{phiyphim}, we have 
\begin{align}
\label{nabla2yphim}
{}^g\nabla^2_{ij}Y_{\phi}=-(P_\phi)_{ij;}{}^k(\nabla_k\phi)-(P_\phi)_{ik}(\nabla^2\phi)^k_j,    
\end{align}
where $\nabla^2\phi=P_\phi+\lambda g$.

Moreover, from \cref{connection-conformally-flat}, we have 
\begin{equation}
\label{nabla2gu}
{}^{g_u}\nabla_{ij}^2(\phi-uY_{\phi})={}^g\nabla^2_{ij}(\phi-uY_{\phi})
-u(V_\phi)_l^k(P_\phi)_{i,j}^l\nabla_k(\phi-uY_{\phi}).    
\end{equation}
Putting together \cref{nabla2yphim,nabla2gu,u-vrelations}, we get $(\widetilde{\nabla}^2\widetilde{\phi}-\lambda g)_{ij}=0$.

We now prove that $\widetilde{F}_\phi=0$. Since $(\widetilde{\nabla}^2\widetilde{\phi}-\lambda g)_{IJ}=0$, we compute that $\partial_u\widetilde{F}_\phi=0$. Hence, it suffices to check that $\restr{\widetilde{F}_\phi}{u=0}=0$, which is done by using \cref{g'phi'} and direct computation.
\end{proof}
\subsection{Einstein manifolds}
\begin{proof}[Proof of \cref{second-examples-theorem}]
For Einstein manifolds of the form of \cref{einstein-manifold-equation}, \cref{ijcomponent,scomponent} can be written as 
\begin{equation}
\label{ambient-equations-for-solitons}
\begin{aligned}
2\mu g_{ij}-\frac{1}{2}g'_{ij} =0\bmod O(u^\infty),\\
 -\frac{1}{2}g^{ij}g'_{ij}+2\phi'=0\bmod O(u^\infty).
\end{aligned}    
\end{equation}
We see from direct computation, and \cref{normal-ambient-metric-exists}, that
\begin{equation}
\label{einstein-ambient-metric}
g_u=(1+4\mu u)g,\quad \phi_u=\frac{n}{4}\mathrm{ln}|1+4\mu u|+c  \end{equation} is the unique solution to \cref{ambient-equations-for-solitons} with the boundary conditions $(g_0,\phi_0)=(g,\phi)$. 
Also, we deduce from \cref{einstein-ambient-metric} that
\begin{equation*}
   v_{k,\phi}=\frac{1}{k!}\mu^k \prod_{j=0}^{k-1}(n-4j).\qedhere
\end{equation*}\
\end{proof}

\section{Weighted renormalized volume coefficients of manifolds with density}
\label{sec:renormalized-volume-coefficients}
In this section, we study weighted renormalized volume coefficients, and prove \cref{linear-expansion-theorem,first-derivative-fphim,second-derivative-theorem,infinitesimal-conformal-change-theorem}.

We define 
$$ \mathcal{W}_{k,\phi} :=\int_M  \tau^k v_{k,\phi} \, (4\pi\tau)^{-n/2}e^{-\phi}\mathrm{dvol}_g.$$
We now write down the renormalized volume coefficients of gradient Ricci solitons~(cf.\ \cite[Lemma 2.3]{Case2016v}).
\begin{lemma}
For gradient Ricci solitons, we have
\begin{equation}
\label{soliton-power-formula}    
v_{k,\phi}=\frac{1}{k!}(v_{1,\phi})^k.
\end{equation}
\end{lemma}
\begin{proof}
We deduce from \cref{examples-theorem} that for gradient Ricci solitons, \[g_u=g,\quad \phi_u=\phi-u Y_\phi,\] where $Y_\phi$ is constant. Hence, $v_{1,\phi}=\frac{1}{2}g^{ij}g'_{ij}-\phi'=Y_\phi$. \cref{volcoefficientsdef} implies \cref{soliton-power-formula}.
\end{proof}
We know from \cref{gradient-soliton-g'phi'} that $v_{1,\phi}=-\frac{1}{2}(n\lambda-c)$.

We now define the tensor $(\Omega_\phi)^{(k)}_{ij}$, which is analogous to the weighted version of the extended obstruction tensor~(cf. \cites{CaseKhaitan2022}).
\begin{definition}
\label{definition-extended-obstruction-tensors}
Let $k\geq 1$. The tensor $(\Omega_\phi)^{(k)}_{ij}$ is defined as \begin{equation*}(\Omega_\phi)^{(k)}_{ij}:= \restr{\widetilde{R}_{\infty ij\infty,\underbrace{\scriptstyle \infty\dotsm\infty}_{k-1}}}{u=0}. \end{equation*}
\end{definition}
\begin{proof}[Proof of \cref{linear-expansion-theorem}]
We perform our calculations at an arbitrary $u$. Let $(\Lambda_u^{(k)})_{i j}=\widetilde{R}_{\infty i j \infty,\underbrace{\scriptstyle\infty \ldots \infty}_{k-1}}$. Hence, $\restr{\Lambda^{(k)}_u}{{u=0}}=(\Omega_\phi^m)^{(k)}$.
The $k=1$ case is trivially true. Let us assume that these formulas hold true for $(\partial^{k-1} g_{ij},\partial^{k-1} \phi)$, and consider $(\partial^{k} g_{ij},\partial^{k} \phi)$. For brevity, we refer to $\Lambda_u,\phi_u$ and $g_u$ as $\Lambda,\phi$ and $g$ respectively in this proof.

For a normal weighted ambient space, we have
\begin{subequations}
\begin{align}
\begin{split}
\label{lambda-derivative}
\partial_{u} \Lambda_{i j}^{(k)}&=\Lambda_{i j}^{(k+1)}+g^{l m} g_{m(i}^{\prime} \Lambda_{j) l}^{(k)},
\end{split}\\
\begin{split}
\label{g-double-prime}
g^{\prime\prime}_{ij}&=2\Lambda_{ij}^{(1)}+\frac{1}{2}g^{kl}g^\prime_{ik}g^\prime_{jl},
\end{split}\\
\begin{split}
\label{g-inverse}
(g^{ij})^\prime&=-g^{ik}g^\prime_{kl}g^{lj},
\end{split}\\
\begin{split}
\label{phi-double-prime}
\phi^{\prime\prime}&=g^{ij}\Lambda_{ij}^{(1)}.
\end{split}
\end{align}
\end{subequations}

\cref{lambda-derivative} is derived using \cref{christoffel symbols for the undetermined metric}. 

\cref{g-double-prime} is derived using \cref{rinftyinfty}.

\cref{phi-double-prime} is derived using the fact that $(\widetilde{P}_\phi)_{\infty\infty}=O(u^\infty)$, \cref{christoffel symbols for the undetermined metric} and \cref{rinftyinfty}.

From \cref{lambda-derivative,g-double-prime,phi-double-prime,g-inverse}, we get
\begin{equation*}
\begin{aligned}
\partial^{k} g_{ij}&=2\Lambda_{ij}^{(k-1)}+\mathcal{G}^{(k)}(g^\prime,\Lambda^{(1)}_{ij},\dotsm,\Lambda^{(k-2)}_{ij}),\\
\partial^{k} \phi&=g^{ij}\Lambda_{ij}^{(k-1)}+\mathcal{F}^{(k)}(\phi^\prime,g^\prime,\Lambda^{(1)}_{ij},\dotsm,\Lambda^{(k-2)}_{ij}),\\
\frac{1}{2}g^{ij}\partial^{k} g_{ij}-\partial^{k} \phi&=\mathcal{L}^{(k)}(\phi^\prime,g^\prime,\Lambda^{(1)}_{ij},\dotsm,\Lambda^{(k-2)}_{ij}).
\end{aligned}
\end{equation*}
Thus, by induction, we complete the proof of \cref{linear-formulas,v-equation}.

Now, upon studying \cref{volcoefficientsdef}, we observe that each $ v_{k,\phi} $ can be written as a complete contraction of $\partial_u^l g_{ij}$ and $\partial_u^l \phi$ for $1\leq l\leq k-1$ and $\frac{1}{2} g^{i j} \partial^{k} g_{i j}-\partial^{k} \phi$. Hence, for $k\geq 1$, we have $$
v_{k,\phi}(g)=\mathcal{V}^{(k)}\left(Y_{\phi}, P_\phi, (\Lambda_\phi)^{(1)}, \ldots, (\Lambda_\phi)^{(k-2)}\right)
$$ for some linear combination of contractions $\mathcal{V}$.
Restricting to $u=0$ and recalling that $\restr{\Lambda^{(k)}_{ij}}{{u=0}}=(\Omega_\phi)_{ij}^{(k)}$ completes the proof of \cref{linear-expansion-theorem}.
\end{proof}
Using \cref{linear-expansion-theorem}, we show that the conformal transformation formulas of
$\partial^k_{u}g_{ij},\partial^k_{u}\phi$ and $ v_{k,\phi} $ involve at most the second derivatives of $\omega$.
\begin{lemma}
Let $k\geq 1$.
Under conformal change, the conformal transformation laws of $\partial_{u}^{k} g$ and  $\partial_{u}^{k} \phi$ involve at most the second derivatives of $\omega .$
\end{lemma}
\begin{proof}
Using \cref{conformalchangedefinition}, we observe that under conformal change,
\begin{equation}
\label{rconformaltransformationformula}
\widehat{\widetilde{R}}_{ABCD,M_1\dotsm M_r}=\widetilde{R}_{ABCD,M_1\dotsm M_r}.    
\end{equation}
Hence, we complete the proof using \cref{conformaltransformationformulas,rconformaltransformationformula,linear-expansion-theorem}.
\end{proof}


\begin{proof}[Proof of \cref{infinitesimal-conformal-change-theorem}]
Let $\man$ be a manifold with density, and let $(\widetilde{g},\widetilde{\phi})$ be a weighted ambient metric in normal form with respect to $(g,\phi)$, such that $(g_0,\phi_0)=(g,\phi)$. If we choose another metric measure structure $(\widehat{g},\widehat{\phi})=(g,\phi+\omega)$ in the same conformal class $[g,\phi]$, the weighted ambient space in normal form for it is $(2dudv + \widehat{g}_u,\widehat{\phi}_u-v+\lambda uv)$ such that $(\widehat{g}_0,\widehat{\phi}_0)=(\widehat{g},\widehat{\phi})$. There exists a diffeomorphism $\psi:\cmG\to \cmG$ such that 
\begin{align*}
\psi^*[2dudv + g_u]&=\widehat{g}_u+2dudv,\\
\psi^*[\phi_u-v+\lambda uv]&=\widehat{\phi}_u-v+\lambda uv.
\end{align*}
Let us now consider a one-parameter family of conformal transformations $(\widehat{g}(t),\widehat{\phi}(t))=(g, \phi + t\omega)$. This generates a one-parameter family of diffeomorphisms $\psi^*_t$ such that
\begin{equation}
\label{pullback-normal-form}
\begin{aligned}
\psi_t^*[2dudv + g_u]&=\widehat{g}_u+2dudv,\\
\psi_t^*[\phi_u-v+\lambda uv]&=\widehat{\phi}_u-v+\lambda uv.
\end{aligned} 
\end{equation}
This one-parameter family of diffeomorphisms also generates a vector field $X=X^0\partial_v + X^i\partial_i+X^\infty \partial_\infty$ such that the flow along $X$ corresponds to the family of diffeomorphisms. 
We write \cref{pullback-normal-form} as
\begin{equation}
\begin{aligned}
\label{lxg}
L_X(2dudv+g_u)&=\delta g_u,\\
L_X (\phi_u-v+\lambda uv)&=\delta\phi_u.
\end{aligned}
\end{equation}
On matching coefficients in \cref{lxg}, we get
\begin{equation}
\label{coeff}
\begin{aligned}
\partial_0(X^\infty)&=0,\\
\partial_0(X^0)+\partial_\infty(X^\infty)&=0,\\
\partial_\infty(X^0)&=0,\\
g_{ij}\partial_0(X^j)+\partial_i(X^\infty)&=0,\\
g_{ij}\partial_\infty(X^j)+\partial_i(X^0)&=0.
\end{aligned}    
\end{equation}
The boundary conditions are 
\begin{equation}
\label{bdryconditions}
\begin{aligned}
\restr{X^0}{u=0}&=-\omega,\\
\restr{X^i}{u=0}&=0,\\
\restr{X^\infty}{u=0}&=0.
\end{aligned}    
\end{equation}
From \cref{coeff,bdryconditions}, we get 
\begin{equation}
\label{xvalues}
X^0=-\omega,\quad X^i=\int_{0}^{u} g^{ij}(s) ds\,\partial_j\omega,\quad X^\infty=0.   
\end{equation}
From \cref{lxconditions,xvalues}, we deduce that
\begin{equation*}
\begin{aligned}
\delta g_{ij}=2\nabla_{(i}X_{j)},\quad \delta\phi=X^i\partial_i\phi_u+\omega-\lambda \omega u.
\end{aligned}    
\end{equation*}
We have
\begin{equation*}
\begin{aligned}
\frac{\delta v_{\phi}}{v_{\phi}}&=\delta \ln v_{\phi}\\
&=-\delta\phi_u+\omega+\frac{1}{2}g^{ij}\delta g_{ij}\\
&=\omega \lambda u +\nabla_i^{(u)} X^i-X^i\partial\phi_u.
\end{aligned}
\end{equation*}
Therefore, 
\begin{equation*}
\delta v_{\phi}=v_{\phi}\omega \lambda u +v_{\phi}\nabla_i^{(u)} X^i-v_{\phi}X^i\partial\phi_u.
\end{equation*}
We observe that
\begin{equation*}
v_{\phi}\nabla_i^{(u)} X^i=\nabla_i^{(0)} (v_{\phi} X^i)+\partial_i(\phi_u-\phi_0)v_{\phi} X^i.    
\end{equation*}
Hence, we have
\begin{equation}
\label{nabla*}
\delta v_{\phi}=\lambda u \omega v_{\phi}+\nabla_i^*{}^{(0)}(v_{\phi} X^i),   
\end{equation}
where $-\nabla^*$ is the adjoint of $\nabla$ with respect to the $L^2$-inner product induced by $e^{-\phi}\mathrm{dvol}_g$. 

From \cref{nabla*}, we deduce that for $k\geq 1$,
\begin{equation}
\label{variation of volumetric coefficient}
\delta v_{k,\phi}=\omega\lambda v_{k-1,\phi} +\frac{1}{ k !} \restr{\partial_{u}^{k}}{u=0}[{\nabla^*_i}^{(0)}(v_{\phi} X^i)],
\end{equation}
where $v_{0,\phi}=1$.
We may write \cref{variation of volumetric coefficient} as 
\begin{equation*}
\delta v_{k,\phi}=\omega\lambda v_{k-1,\phi} + {\nabla^*_i}^{(0)}[(L_{k,\phi})^{ij}\nabla_j \omega], \end{equation*}
where 
\begin{equation}
\label{l-definition}
(L_{k,\phi})^{ij}=\frac{1}{k!}\restr{\partial^k_u}{u=0}[v_{\phi}\int_0^u g^{ij}(u)\,du].\qedhere
\end{equation}
\end{proof}
\begin{corollary}
As ${\nabla^*_i}^{(0)}[(L_{k,\phi})^{ij}\nabla_j \omega]$ is a divergence term, we have 
\begin{equation}
\label{variational-integral}
\int_M \delta v_{k,\phi} \,(4\pi\tau)^{-n/2}e^{-\phi}\mathrm{dvol}_g=\int_M \omega\lambda v_{k-1,\phi} \,(4\pi\tau)^{-n/2}e^{-\phi}\mathrm{dvol}_g. 
\end{equation}
\end{corollary}

\subsubsection{Critical points}
\label{sec:critical-points}
In this section, we study the critical points of $ \mathcal{W}_{k,\phi}  $. First, we show that the critical points of $ \mathcal{W}_{k,\phi}  $ in $\mathcal{C}_1$ satisfy the condition that $ v_{k,\phi} $ is constant.
\begin{proof}[Proof of \cref{first-derivative-fphim}]
\label{proof-of-first-derivative-fphim}
We have 
\begin{equation*}
\delta  \mathcal{W}_{k,\phi}  =\int_M \tau^k \delta v_{k,\phi} \,(4\pi\tau)^{-n/2}e^{-\phi}\mathrm{dvol}_g-\int_M \tau^k v_{k,\phi} \omega\, (4\pi\tau)^{-n/2}e^{-\phi}\mathrm{dvol}_g.
\end{equation*}
From \cref{variational-integral}, we get
\begin{equation}
\label{variation-fphim}
\delta \mathcal{W}_{k,\phi}  =\int_M \tau^k \omega (\lambda v_{k-1,\phi} - v_{k,\phi})  \,(4\pi\tau)^{-n/2}e^{-\phi}\mathrm{dvol}_g. \end{equation}
Let $$\mathrm{WtdVol}:=\int_M (4\pi\tau)^{-n/2}e^{-\phi}\mathrm{dvol}_g$$ denote the weighted volume functional. Using the Lagrange multiplier method, we have for some $a>0$,
\begin{align*}
\delta\left[\mathcal{W}_{k,\phi} -a \mathrm{WtdVol}(M)\right](g,\phi)=0.\end{align*}
Hence, \begin{multline*}
\int_M \tau^k \omega (\lambda v_{k-1,\phi} - v_{k,\phi})  \,(4\pi\tau)^{-n/2}e^{-\phi}\mathrm{dvol}_g\\+ a\int_M \omega \,(4\pi\tau)^{-n/2}e^{-\phi}\mathrm{dvol}_g=0.
\end{multline*} 
We see that $\delta \mathcal{W}_{k,\phi}=0$ iff $ v_{k,\phi}-\lambda v_{k-1,\phi}=a\tau^{-k}$.
\end{proof}

We now show that gradient Ricci solitons are local extrema of $\mathcal{W}_{k,\phi} $ in $\mathcal{C}_1$. Note that the tangent space of $\mathcal{C}_1$ at $(g,\phi)$ is \begin{equation*}
\label{tangent-space-c1-definition}
    \mathcal{T}_{(g,\phi)}\mathcal{C}_1=\{\omega:\int_M \omega \,(4\pi\tau)^{-n/2}e^{-\phi}\mathrm{dvol}_g=0\}.
\end{equation*}
\begin{lemma}
Gradient Ricci solitons are critical points of $\mathcal{W}_{k,\phi} $ in $\mathcal{C}_1$.
\end{lemma}
\begin{proof}
For shrinking gradient Ricci solitons, we have from \cref{gradient-soliton-g'phi',soliton-power-formula} that \[v_{k,\phi}=\frac{1}{k!}\left[-\frac{1}{2}(n\lambda - c)\right]^k\] for $k\geq 0$. The conclusion follows from \cref{first-derivative-fphim}.
\end{proof}

Now we write down the formula for $(\left.\mathcal{W}_{k,\phi} \right|_{\mathcal{C}_{1}})^{\prime \prime}(\omega)$. 
\begin{lemma}
Let $\man$ be a connected compact manifold with density and suppose $(g,\phi)$ is a critical point of $\mathcal{W}_{k,\phi}$. Let $\omega \in C^{\infty}(M)$ satisfy $\int_{M} \omega \,(4\pi\tau)^{-n/2}e^{-\phi}\mathrm{dvol}_g=0$. Then
\begin{multline}
\label{double-derivative}
(\left.\mathcal{W}_{k,\phi} \right|_{\mathcal{C}_{1}})^{\prime \prime}(\omega)=\int_M\Big[(\lambda v_{k-2,\phi}- v_{k-1,\phi})\lambda\omega^2\\ +\big[-\lambda (L_{k-1,\phi})^{ij}+
(L_{k,\phi})^{ij}\big]\nabla_i \omega\nabla_j \omega\Big]\, (4\pi\tau)^{-n/2}e^{-\phi}\mathrm{dvol}_g,    
\end{multline}
where $v_{-1,\phi}=0$.
\end{lemma}
\begin{proof}
Let $\gamma_t$ be a curve in $\mathcal{C}_1$ such that $\gamma_0=(g,\phi)$ and $\gamma^\prime_0=\omega$. Then 
\begin{equation*}
(\left.\mathcal{W}_{k,\phi} \right|_{\mathcal{C}_{1}})^{\prime \prime}(\omega)=\left.\partial_{t}^{2}\right|_{t=0}\left((\mathcal{W}_{k,\phi} )-a \mathrm{WtdVol}(M)\right)\left(g_{t},f_t\right).  \end{equation*}
On using \cref{variation-fphim} and integrating by parts, we get \cref{double-derivative}.
\end{proof}

\begin{proof}[Proof of \cref{second-derivative-theorem}]
Using \cref{l-definition} and that fact that $g^{ij}=g^{ij}_{(0)}$ for gradient Ricci solitons, we calculate that
\begin{equation}
\label{l-values}
\begin{aligned}
(L_{0,\phi})^{ij}&=\restr{ug_{(0)}^{ij}}{u=0}=0,\\
(L_{k,\phi})^{ij}&=\frac{1}{(k-1)!}\left[-\frac{1}{2}(n\lambda - c)\right]^{k-1}g^{ij}\text{ for }k\geq 1,
\end{aligned}  
\end{equation}
where $c$ is such that $\Delta\phi-|\nabla\phi|^2+2\lambda\phi=c$.
Now note that for a gradient soliton, we have $\lambda v_{-1,\phi}-v_{0,\phi}=-1$, and by \cref{first-derivative-fphim} we have $\lambda v_{k-1,\phi}-v_{k,\phi}=0$ for $k\geq 1$. Hence, using \cref{double-derivative,l-values}, we calculate that at $u=0$,
\begin{subequations}
\begin{align}
\begin{split}
\label{1phiinfty''}
(\mathcal{W}_{1,\phi})^{\prime\prime}(\omega)&=- \int_M (\lambda \omega^2-|\nabla\omega|^2) (4\pi\tau)^{-n/2}e^{-\phi}\mathrm{dvol}_g\geq 0,
\end{split}\\
\begin{split}
\label{2phiinfty''}
(\mathcal{W}_{2,\phi})^{\prime\prime}(\omega)&=b_2 \int_M (\lambda\omega^2-|\nabla \omega|^2)(4\pi\tau)^{-n/2}e^{-\phi}\mathrm{dvol}_g\leq 0.
\end{split}
\end{align}    
\end{subequations}
where $b_2=(\lambda+2^{-1}(n\lambda - c))>0$.
\cref{1phiinfty'',2phiinfty''} follow from \cite[Proposition (9.1)]{Case2014sd}.
In general, we have
\begin{equation}
\label{k-general-case}
(\mathcal{W}_{k,\phi})^{\prime\prime}(\omega)=b_k \int_M (\lambda\omega^2-|\nabla \omega|^2)(4\pi\tau)^{-n/2}e^{-\phi}\mathrm{dvol}_g    
\end{equation}
for $k\geq 2$, where 
\begin{equation*}
b_k=\frac{1}{(k-2)!}\left[-\frac{1}{2}(n\lambda - c)\right]^{k-2}\left[\lambda+\frac{1}{2(k-1)}(n\lambda - c)\right]. \end{equation*}
We see from \cref{1phiinfty'',k-general-case} that for all $k\in\mathbb{N}$, we have
\[(-1)^k (\mathcal{W}_{k,\phi})^{\prime\prime}(\omega)\leq 0\] for shrinking solitons. 
\end{proof}

\bibliographystyle{plain}
\bibliography{bib.bib}

\end{document}